\documentclass[11pt]{amsart}

\usepackage[usenames,dvipsnames,svgnames,table]{xcolor}
\usepackage[colorlinks=true, pdfstartview=FitV, linkcolor=blue, citecolor=blue, urlcolor=blue]{hyperref}

\usepackage{geometry}                
\usepackage{graphicx}
\usepackage{amssymb}
\usepackage{epstopdf}
\usepackage{lscape}
\usepackage[utf8]{inputenc}
\usepackage{tikz,caption}
\usepackage{enumitem}
\setlist[itemize]{leftmargin=2em}
\setlist[enumerate]{leftmargin=2em}
\usepackage{booktabs}
\usepackage{multirow}
\usepackage{mathtools}

\definecolor{darkblue}{rgb}{0.0,0,0.7} 
\definecolor{darkred}{rgb}{0.7,0,0} 
\definecolor{darkgreen}{rgb}{0, .6, 0} 

\newcommand{\defncolor}{\color{darkred}}
\newcommand{\defn}[1]{{\defncolor\emph{#1}}} 

\newtheorem{theorem}{Theorem}[section]
\newtheorem{prop}[theorem]{Proposition}
\newtheorem{cor}[theorem]{Corollary}
\newtheorem{lemma}[theorem]{Lemma}

\theoremstyle{definition}
\newtheorem{definition}[theorem]{Definition}
\newtheorem{example}[theorem]{Example}
\newtheorem{remark}[theorem]{Remark}
\numberwithin{equation}{section}

\newcommand{\idiot}[1]{\vspace{5 mm}\par \noindent
\marginpar{\textsc{Note}}
\framebox{\begin{minipage}[c]{0.95 \textwidth}
#1 \end{minipage}}\vspace{5 mm}\par}

\renewcommand{\idiot}[1]{}

\def\NN{{\mathbb N}}
\def\CC{{\mathbb C}}

\def\mvdash{{\,\vdash\!\!\vdash}}

\def\dcl{{\{\!\!\{}}
\def\dcr{{\}\!\!\}}}

\def\o{\overline}

\newcommand{\End}{\operatorname{End}}
\newcommand{\Hom}{\operatorname{Hom}}

\newcommand{\SSYT}{\operatorname{SSYT}}
\newcommand{\SYT}{\operatorname{SYT}}
\newcommand{\SSMT}{\operatorname{SSMT}}
\newcommand{\SMT}{\operatorname{SMT}}
\newcommand{\cells}{\operatorname{cells}}

\newcommand{\cf}{\textit{cf.} }
\newcommand{\ie}{\textit{i.e.}}
\newcommand{\content}{\operatorname{content}}

\newcommand{\smallstirling}[2]{\left\{\!\begin{smallmatrix}#1\\#2\end{smallmatrix}\!\right\}}
\newcommand{\smallbinom}[2]{\left(\!\strut\begin{smallmatrix}#1\\#2\end{smallmatrix}\!\right)}
\newcommand\compactcdots{\makebox[1em][c]{.\hfil.\hfil.}}

\usepackage[colorinlistoftodos]{todonotes}

\newdimen\squaresize \squaresize=10pt
\newdimen\thickness \thickness=0.4pt

\def\square#1{\hbox{\vrule width \thickness
     \vbox to \squaresize{\hrule height \thickness\vss
        \hbox to \squaresize{\hss#1\hss}
     \vss\hrule height\thickness}
\unskip\vrule width \thickness}
\kern-\thickness}

\def\vsquare#1{\vbox{\square{$#1$}}\kern-\thickness}

\def\young#1{
\vbox{\smallskip\offinterlineskip
\halign{&\vsquare{##}\cr #1}}}

\def\thisbox#1{\kern-.09ex\fbox{#1}}
\def\downbox#1{\lower1.200em\hbox{#1}}
\newdimen\Squaresize \Squaresize=20pt
\newdimen\Thickness \Thickness=0.4pt

\def\Square#1{\hbox{\vrule width \Thickness
     \vbox to \Squaresize{\hrule height \Thickness\vss
        \hbox to \Squaresize{\hss#1\hss}
     \vss\hrule height\Thickness}
\unskip\vrule width \Thickness}
\kern-\Thickness}

\def\Vsquare#1{\vbox{\Square{$#1$}}\kern-\Thickness}

\title[Multiset partition insertion]{An insertion algorithm on multiset partitions with applications to diagram algebras}

\author[Colmenarejo]{Laura Colmenarejo}
\address[L. Colmenarejo]{Department of Mathematics and Statistics, UMass Amherst, 710 N Pleasant St, Amherst, U.S.A}
\email{laura.colmenarejo.hernando@gmail.com}
\urladdr{https://sites.google.com/view/l-colmenarejo/home}

\author[Orellana]{Rosa Orellana}
\address[R. Orellana]{Mathematics Department, Dartmouth College, 6188 Kemeny Hall,
Hanover, NH 03755, U.S.A.}
\email{Rosa.C.Orellana@dartmouth.edu}
\urladdr{https://math.dartmouth.edu/~orellana/}

\author[Saliola]{Franco Saliola}
\address[F. Saliola]{D\'epartement de math\'ematiques,
Universit\'e du Qu\'ebec \`a Montr\'eal, Canada}
\email{saliola.franco@uqam.ca}
\urladdr{http://lacim.uqam.ca/~saliola/}

\author[Schilling]{Anne Schilling}
\address[A. Schilling]{Department of Mathematics, University of California, One Shields
Avenue, Davis, CA 95616-8633, U.S.A.}
\email{anne@math.ucdavis.edu}
\urladdr{http://www.math.ucdavis.edu/\~{}anne}

\author[Zabrocki]{Mike Zabrocki}
\address[M. Zabrocki]{Department of Mathematics and Statistics,  York University, 4700 Keele Street, Toronto, 
Ontario M3J 1P3, Canada}
\email{zabrocki@mathstat.yorku.ca}
\urladdr{http://garsia.math.yorku.ca/~zabrocki/}

\begin{document}

\begin{abstract}
We generalize the Robinson--Schensted--Knuth algorithm to the insertion of two row arrays of 
multisets. This generalization leads to new enumerative results that have representation theoretic 
interpretations as decompositions of centralizer algebras and the spaces they act on.  In addition, 
restrictions on the multisets lead to further identities and representation theory analogues.    
For instance, we obtain a bijection between words of length $k$ with entries in $[n]$ and pairs of 
tableaux of the same shape with one being a standard Young tableau of size $n$ and the other
being a standard multiset tableau of content $[k]$. We also obtain an algorithm
from partition diagrams to pairs of a standard tableau and a standard
multiset tableau of the same shape, which has the remarkable property that it is
well-behaved with respect to restricting a representation to a subalgebra. This
insertion algorithm matches recent representation-theoretic results of
Halverson and Jacobson \cite{HJ}.
\end{abstract}

\maketitle

\section{Introduction}
\label{Introduction}

We explore a variant of the Robinson--Schensted--Knuth (RSK) algorithm, where we insert
multisets instead of integer entries.  If we restrict the multisets to all have size one,
the algorithm we are using is the usual RSK algorithm.
Applying this insertion to different arrays of multisets gives
rise to a purely enumerative result that is a combinatorial
manifestation of a double centralizer theorem from representation theory.
Although representation theory serves as a principal motivation for studying
these algorithms, no familiarity is assumed in our exposition of the
enumerative and combinatorial results.

\medskip

The RSK algorithm evolved over the last century from a procedure defined on
permutations (in the work of Robinson \cite{Robinson}) to a procedure defined
on finite sequences of integers (in the work of Schensted \cite{Schensted}) and
finally to a procedure defined on ``generalized permutations'' by Knuth
\cite{Knuth}. In each of these versions, the algorithm establishes
a correspondence between the initial input and pairs of combinatorial objects
called tableaux subject to certain constraints (see
Section~\ref{section.notation} for definitions).

Each of the above procedures reflects a classical direct-sum decomposition
result in representation theory.
While the reader will find more details in
Section~\ref{section.classical-RSK-rep-theory}, we present here
an overview.
Broadly speaking, we start with two families of operators,
say $\mathcal A$ and $\mathcal B$, acting on a vector space $V$,
and we determine the finest decomposition of $V$
into a direct sum of subspaces that are invariant for all the operators,
say $V = \bigoplus_{\lambda \in \Lambda} V^\lambda$ for some indexing
set $\Lambda$.
Under certain circumstances, the actions of $\mathcal A$ and $\mathcal B$
neatly separate the subspaces $V^\lambda$;
more precisely, $V^\lambda$ can be
expressed as a tensor product $U^\lambda \otimes W^\lambda$,
where the action of $\mathcal A$ only affects $U^\lambda$
and the action of $\mathcal B$ only affects $W^\lambda$.
Thus, we obtain the decomposition:
\begin{equation*}
    V \cong \bigoplus_{\lambda \in \Lambda} \left(U^\lambda \otimes W^\lambda\right).
\end{equation*}
At the combinatorial level, this decomposition implies that there is a bijection
between $\mathcal V$ and
    $\bigcup_{\lambda \in \Lambda}
    \left(\mathcal U^\lambda \times \mathcal W^\lambda\right)$,
where $\mathcal V$, $\mathcal U^\lambda$ and
$\mathcal W^\lambda$ denote bases of $V$, $U^\lambda$ and $W^\lambda$,
respectively.

One example of this is when $GL(V)$ and $S_k$ both act on $V^{\otimes k}$
(see Section~\ref{section.classical-RSK-rep-theory} for details).
In this case, we deduce the existence of a bijection between the set of finite
sequences of length~$k$ with entries in $\{1, 2, \ldots, \dim(V)\}$ (\ie,
a basis of $V^{\otimes k}$) and
the union of the set of pairs consisting of a semistandard tableau of shape
$\lambda$ with entries in $\{1, 2, \ldots, \dim(V)\}$
(\ie, a basis of $U^\lambda$)
and a standard tableaux of shape $\lambda$ and size $k$
(\ie, a basis of $W^\lambda$).
This is precisely what the RSK algorithm does (see Section~\ref{section.rsk}).

The above situation also holds for many other pairs of families of operators
acting on $V^{\otimes k}$; for instance: the partition algebra and the
symmetric group; the Brauer algebra and the orthogonal group; and the Hecke
algebra and the quantum group of type~$A$.

\medskip

In this paper, we adapt the RSK algorithm to the insertion of arrays of multisets.  
This adaptation gives combinatorial descriptions of other direct-sum decomposition
results in representation theory.  Furthermore, restrictions on the multisets result
in a bijection and an enumerative result relating sets of combinatorial objects.
For instance, by considering a vector space on which
both the symmetric group and the partition algebra act, we obtain a bijection
between words of length $k$ with entries in $[n]$ and pairs of tableaux of the
same shape with one being a standard Young tableau of size $n$ and the other
being a standard multiset tableau of content $[k]$.  We also obtain an algorithm
from monomials in a polynomial ring to pairs of a standard tableau and a standard
multiset tableau of the same shape and from elements of diagram algebras to pairs
of standard multiset tableaux.

Note that algorithms that relate partition diagrams and pairs of paths
in the Bratteli diagram for the partition algebras 
have been known since the late 1990s \cite{HalLew, MarRol}. 
These paths are referred as ``vacillating tableaux''
and they are analogues of a path in the Young's lattice,
which is the Bratteli diagram for the symmetric groups.
Paths in the Young lattice are encoded by standard Young tableaux.

Recently, a new combinatorial interpretation for the dimensions of the
irreducible representations for the partition algebra has appeared in the
literature \cite{BH,BHH,OZ,HJ,Halverson.2019}. In particular, Benkart and Halverson \cite{BH}
presented a bijection between vacillating tableaux and ``set-partition tableaux''
(tableaux whose entries are sets of positive integers).
There are two main advantages to working with set-partition tableaux instead
of vacillating tableaux. Firstly, they are closer in spirit to the ubiquitous
Young tableaux. Secondly, the definition extends naturally to the notion of
multiset tableaux (tableaux whose entries are multisets of positive integers)
and working with multiset tableaux leads to new enumerative and algebraic
results that are not obvious by other means (see Proposition~\ref{prop:twolinenewRSK},
Corollary~\ref{prop:enumnpk}, and Theorem~\ref{thm:settableaux}).

Our insertion algorithm from partition diagrams to pairs of a standard tableau and a standard
multiset tableau of the same shape has the remarkable property that it is well-behaved
with respect to the subalgebra structure of the partition algebra.
One surprising consequence is that we are able to provide explicit combinatorial
descriptions of the sets of tableaux that give the dimensions of the irreducible representations 
associated to the prominent subalgebras of the partition algebras, such
as the symmetric group, the Brauer algebra, the rook algebra, the rook-Brauer algebra, the Temperley--Lieb 
algebra, the Motzkin algebra, the planar rook algebra, and the planar algebra (see Lemma~\ref{lemma:tableauxdesc}).
This gives rise to analogues of the famous identity $n! = \sum_{\lambda \vdash n} (f^{\lambda})^2$ for the symmetric
group, where $f^\lambda$ is the number of standard tableaux of shape $\lambda$,
to all of the above mentioned algebras (see Corollary~\ref{alg-dim}). We prove that the dimensions of the
irreducible representations of the various algebras
is equal to the number of our combinatorially-defined tableaux
by establishing that the branching rules are encoded in the tableaux
(see Section~\ref{section.bratteli} and Corollary~\ref{irred-dim}).
Our insertion is different from combining the insertion of Halverson and Lewandowski~\cite{HalLew}
from partition diagrams to paths in the Bratteli diagram with the bijection of Halverson and Benkart~\cite{BH}
from paths in the Bratteli diagram to set-partition tableaux (see Section~\ref{section.bratteli}).

\medskip

The paper is organized as follows.
In Section~\ref{section.notation}, we define the principal combinatorial
objects used throughout this paper: multiset tableaux.
In Section~\ref{section.rsk}, we review the RSK algorithm for associating
a pair of tableaux to a generalized permutation.
The above discussion is expanded in Section~\ref{section.classical-RSK-rep-theory} 
by providing more details on how RSK on permutations, words, and generalized
permutations, reflects decomposition results in representation theory.
In Section~\ref{sec:app1}, the RSK algorithm is adapted to the multiset tableaux setting and 
corresponding enumerative results are obtained. The section opens with a description of the enumerative and
combinatorial results and closes by connecting these results with
representation theory.
Finally, in Section~\ref{sec:diagramalgebras} the algorithm is applied to
partition algebra diagrams and it is shown that the new insertion algorithm is well-behaved when
restricted to subalgebras. Finally, the connection to the representation theory
recently developed in~\cite{HJ} is established.

\subsection*{Acknowledgments}
The authors would like to thank BIRS and AIM for the opportunity to collaborate in Banff in April 2018
and in San Jos\'e in November 2018, which greatly facilitated the work on this project. 
We thank Tom Halverson for sharing an early version of~\cite{Halverson.2019} with us.

The second author was partially supported by NSF grant DMS-1700058 and the fourth author 
was partially supported by NSF grants DMS--1760329 and DMS--1764153.  
The third and fifth authors were supported by NSERC Discovery Grants.

\section{Multiset Tableaux}
\label{section.notation}

Throughout this paper, we work with tableaux whose entries are multisets.
Note that any Young tableau---that is, a tableau with integer entries---can be
viewed as a multiset tableau by considering each entry to be a multiset of cardinality $1$.
In this section, we fix notation and define the total orders on multisets that
we use in order to extend the property of being (semi)standard to multiset
tableaux.

\subsection{Partitions}
\label{section.partitions}
A \defn{partition} of $n \in \NN$ is a sequence of
positive integers $\lambda = (\lambda_1, \lambda_2, \ldots,
\lambda_r)$ with $\lambda_1 \geqslant \lambda_2 \geqslant \cdots
\geqslant \lambda_r > 0$ whose sum is $n$.
Note that the empty sequence $()$ is a partition of $0$.
The notation $\lambda \vdash n$ is used to indicate that $\lambda$ is
a partition of $n$.
The \defn{length} of the partition is denoted by $\ell(\lambda) = r$.
As is customary, we depict partitions as diagrams; see Example~\ref{ex:multiset-tableaux}.
The \defn{cells} of the partition are the
coordinates of the boxes in the diagram; that is,
$\cells(\lambda) = \{ (i,j) \mid 1 \leqslant i \leqslant \lambda_j, 1 \leqslant j \leqslant \ell(\lambda) \}$.  
The operation of removing the first row of the partition $\lambda$
is denoted by $\o\lambda=(\lambda_2, \lambda_3, \ldots, \lambda_r)$.

\subsection{Set partitions}
A \defn{set partition} $\pi$ of a set $S$ is a collection $\{\pi_1, \ldots, \pi_k\}$ of non-empty
subsets of $S$ that are mutually disjoint, i.e., $\pi_i\cap\pi_j =\emptyset$ for all $i\neq j$, and $\bigcup_{i} \pi_i = S$. 
The subsets $\pi_i$ are called the \defn{blocks} of the set partition.   We write $\pi\vdash S$ to mean that $\pi$ is a 
set partition of the set $S$. 

\subsection{Multisets}
\label{sec:multisets}

Let $(A, \leqslant_A)$ be a totally ordered set, which we refer to as an
\defn{(ordered) alphabet}.
A \defn{multiset} $S = \dcl a_1, a_2, \ldots, a_r \dcr$ over $A$
is an unordered collection of elements of $A$, allowing repeats.
The collection of multisets
forms an associative monoid with operation
\[
    \dcl a_1, a_2, \ldots, a_r \dcr \uplus
    \dcl b_1, b_2, \ldots, b_d \dcr =
    \dcl a_1, a_2, \ldots, a_r, b_1, b_2, \ldots, b_d \dcr~.
\]
To simplify notation, we let $\dcl {a}^{m_a}, b^{m_b}, c^{m_c}, \ldots \dcr$
denote the multiset that contains $m_a$ occurrences of $a$, $m_b$ occurrences
of $b$, and so on;
for example $\dcl 1^2, 4^3, 5\dcr = \dcl 1, 1, 4, 4, 4, 5\dcr$.

A \defn{multiset partition}\footnotemark{} of a multiset $S$ is a multiset of multisets,
$\pi = \dcl S^{(1)}, S^{(2)}, \ldots, S^{(r)} \dcr$, such that
$S = S^{(1)} \uplus S^{(2)} \uplus \cdots \uplus S^{(r)}$.
We indicate this by the notation $\pi \mvdash S$.
\footnotetext{Multisets are in bijection with integer vectors and multiset partitions
are in bijection with objects known as vector partitions \cite{Comtet, Gessel, MacMahon, Rosas}.
Since integer vectors can be identified with sequences of monomials in
a set of variables, another interpretation for multisets is as monomials in the
variables $\{ x_{c_1}, x_{c_2}, \ldots \}$.}

\subsection{Ordering multisets}
\label{sec:multiset-orders}

We will use two different methods to totally order the collection of all
multisets over an ordered alphabet $A$.
In Section~\ref{sec:app1}, we use \emph{graded lexicographic order}.
If $S = \dcl a_1, \ldots, a_r \dcr$ with $a_1 \leqslant_A \cdots \leqslant_A a_r$
and $S' = \dcl a'_1, \ldots, a'_t \dcr$ with $a'_1 \leqslant_A \cdots \leqslant_A a'_t$,
then we say $S < S'$ in the \defn{graded lexicographic order} if:
\begin{itemize}
    \item
        $r < t$; or

    \item
        $r = t$ and there exists $1 \leqslant i \leqslant t$ such that
        $a_1 = a'_1,
        \ldots,
        a_{i-1} = a'_{i-1}$,
        and
        $a_{i} <_A a'_{i}$.
\end{itemize}
This is a total order \cite{CLO}, with minimum element the empty multiset.

In Section~\ref{sec:diagramalgebras}, where we need only compare disjoint sets,
we use the \emph{last letter order}.
Given two disjoint sets $S$ and $S'$ with elements in an ordered set $A$, we say
$S<S'$ in the \defn{last letter order} if $\mbox{max}(S) <_A \mbox{max}(S')$,
where $\mbox{max}(S)$ is the largest element in $S$.
For example, $\{1,3,5\} < \{2,7\}$.
(This order can be realized as the restriction of a total order on multisets,
for example \emph{reverse lexicographic order}, but this is not
necessary here.)

\subsection{Multiset tableaux}
\label{sec:multiset-tableaux}

Let $\lambda$ be a partition, $A$ an ordered alphabet, and a fixed total order $<$ on multisets (such as the graded
lexicographic order or the last letter order if the multisets are all disjoint sets).
A \defn{semistandard multiset tableau} of \defn{shape} $\lambda$ is a function $T$ that
associates with each cell $(i, j) \in \cells(\lambda)$ a multiset over $A$
such that:
\begin{itemize}
    \item
        $T(i,j) \leqslant T(i,j+1)$ whenever $(i, j)$ and $(i, j+1)$ both belong to $\cells(\lambda)$; and
    \item
        $T(i,j) < T(i+1,j)$ whenever $(i, j)$ and $(i+1, j)$ both belong to $\cells(\lambda)$.
\end{itemize}
The \defn{shape} of a multiset tableau $T$ is the partition $\lambda$,
and the \defn{cells} of $T$ are the cells of its shape.
If $T(i, j) = S$, then we say that $S$ \defn{labels} the cell $(i, j)$,
and that $S$ is an \defn{entry} of $T$.

When drawing multiset tableaux, the multisets are abbreviated as words
without the surrounding multiset delimiters $\dcl, \dcr$, and empty
sets are depicted by blank cells.

\subsection{Content of multiset tableaux}\label{sec:content}
The \defn{content} of a semistandard multiset tableau $T$ is the (disjoint) union of the
entries of $T$. More precisely, the content of $T$ is the multiset
\begin{equation*}
    \content(T) = \biguplus_{(i, j) \in \cells(T)} T(i, j).
\end{equation*}
A semistandard multiset tableau is said to be
\begin{itemize}
    \item
        a \defn{standard multiset tableau} if its content is the set ${\defncolor[k]} := \{1,2, \ldots, k\}$ for some
        $k \in \NN$; in other words, each letter $1,2,\ldots,k$ appears exactly once in the tableau;
    \item
        a \defn{semistandard Young tableau}
        if all its entries are multisets of size $1$;
    \item
        a \defn{standard Young tableau} if it is both standard and
        a semistandard Young tableau.
\end{itemize}
Finally, for a multiset $S$, let
\begin{itemize}
    \item
        \defn{$\SSMT(\lambda, S)$} be the set of \emph{semistandard} multiset tableaux of shape $\lambda$ and content $S$;
    \item
        \defn{$\SMT(\lambda,k)$} be the set of \emph{standard} multiset tableaux of shape $\lambda$ and content $[k]$;
    \item
        \defn{$\SSYT(\lambda, S)$} be the set of semistandard Young tableaux of
        shape $\lambda$ and content $S$; and
    \item
        \defn{$\SYT(\lambda)$} be the set of standard Young tableaux of shape $\lambda$.
\end{itemize}
The \defn{set-partition tableaux} of~\cite[Definition 3.14]{BH} are closely related to our standard multiset tableaux.

\begin{example}
    \label{ex:multiset-tableaux}
    Let $A = \{1, 2, 3, 4, 5\}$ with the usual order on integers.
    Then
    \begin{equation*}
        \squaresize=17pt
        \young{114\cr2&14\cr&2&5\cr}
        \qquad
        \qquad
        \young{3\cr2&3\cr1&1&4\cr}
        \qquad
        \qquad
        \young{134\cr5&26\cr&7&8\cr}
    \end{equation*}
    are three semistandard multiset tableaux of shape $(3, 2, 1)$.
    The leftmost tableau has content $\dcl 1^3, 2^2, 4^2, 5\dcr$.
    The middle tableau has content $\dcl 1^2, 2, 3^2, 4\dcr$
    and is also a semistandard Young tableau.
    The rightmost multiset tableau is standard
    since its content is $\{1, \ldots, 8\}$.
\end{example}

\section{The RSK correspondence}
\label{section.rsk}

We present here the Robinson--Schensted--Knuth (RSK) algorithm for certain finite
sequences in $A \times B$, where $A$ and $B$ are totally ordered sets.
Our presentation is slightly more general than the original \cite{Knuth},
where it was defined on certain finite sequences in $\NN \times \NN$. The
proofs in \cite{Knuth} still hold in this context because they only make
use of the fact that $\NN$ is a totally ordered set.

\medskip

Let $A$ and $B$ be two ordered alphabets.
A \defn{generalized permutation} from $A$ to $B$ is a two-line array of the
form
$\left(\begin{smallmatrix} a_1 & a_2 & \cdots & a_r \\ b_1 & b_2 & \cdots& b_r \end{smallmatrix}\right)$
satisfying:
$a_1, \ldots, a_r \in A$; $b_1, \ldots, b_r \in B$;
$a_i \leqslant_A a_{i+1}$ for $1 \leqslant i \leqslant r-1$;
and $b_{i} \leqslant_B b_{i+1}$ whenever $a_i = a_{i+1}$.

The RSK algorithm constructs two semistandard tableaux of the same shape
from a generalized permutation
$\left(\begin{smallmatrix} a_1 & a_2 & \cdots & a_r \\ b_1 & b_2 & \cdots& b_r \end{smallmatrix}\right)$.
The cells of one of the tableaux are labelled by $a_1, \ldots, a_r$
and the cells of the other tableau are labelled by $b_1, \ldots, b_r$.
The key step in the algorithm is the following insertion procedure.

\begin{definition}[RSK insertion procedure]
    \label{defn:rsk}
    Let $T$ be a semistandard tableau with entries in an ordered set $X$ and $x
    \in X$. The \defn{(row) insertion} of $x$ into $T$ is the tableau
    defined recursively as follows, starting with $r = 1$.
    \begin{enumerate}
        \item
            If $x$ is greater than or equal to all entries of the $r$-th row of
            $T$ (this includes the case where the $r$-th row has no entries),
            then append $x$ to the end of the row; otherwise

        \item
            let $\widehat x$ be the leftmost entry of the row that
            is greater than $x$, replace $\widehat
            x$ with $x$, and insert the $\widehat x$ into row $r+1$
            (\ie, repeat step (1) with $x=\widehat x$ and $r$ replaced by $r+1$).
    \end{enumerate}
\end{definition}

Iterating this insertion procedure allows us to associate two semistandard
tableaux with any generalized permutation
$\left(\begin{smallmatrix} a_1 & a_2 & \cdots & a_r \\ b_1 & b_2 & \cdots& b_r \end{smallmatrix}\right)$
as follows (\cf Example~\ref{ex:rsk}).
\begin{itemize}
    \item
        Start with a pair of empty tableaux;
        \ie, tableaux with no rows and no columns.
    \item
        Insert $b_1$ into the first tableau using the insertion procedure;
        this introduces a new cell, say in position $(i_1, j_1)$;
        add a cell labelled $a_1$ in position $(i_1, j_1)$ of the second tableau.
    \item
        Insert $b_2$ into the first tableau, which introduces a new cell to the
        first tableau; label the corresponding cell in the second tableau by
        $a_2$.
    \item
        Repeat with $(a_3, b_3), (a_4, b_4), \ldots, (a_r, b_r)$.
\end{itemize}
The result is two semistandard tableaux of the same shape,
the first has entries $b_1, \ldots, b_r$,
and the second has entries $a_1, \ldots, a_r$.

\begin{example}
    \label{ex:rsk}
    Consider the alphabets $A = \{ a  <  b  <  c  <  d \}$
    and $B = \{ w  <  x  <  y  <  z \}$.
    The insertion procedure applied to the generalized permutation
    $\left(\begin{smallmatrix}
         a  &  b  &  b  &  c  &  c  &  d  &  d  &  d  \\
         x  &  y  &  z  &  w  &  y  &  w  &  y  &  y
    \end{smallmatrix}\right)$,
    gives the following sequence of tableaux.
    \def\arrow#1#2{{\tiny\xrightarrow{\begin{array}{c}\text{insert~$#1$}\\[-0.3ex]\text{record~$#2$}\end{array}}}}
    \def\myyoung#1{\begingroup\setbox0=\hbox{\young{#1}}\parbox{\wd0}{\box0}\endgroup}
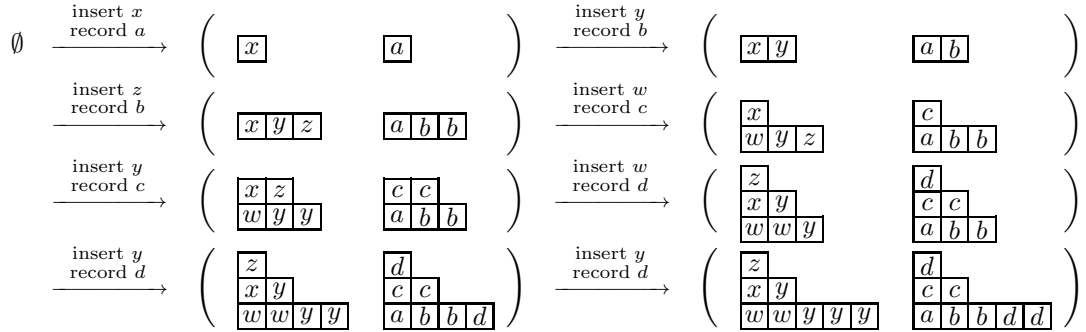
    \begin{equation*}
        \small
        \squaresize=10pt
        \begin{array}{rcrlllcrlll}
            \emptyset
            & \arrow{x}{a} & \bigg( & \hspace{-0.5em} \myyoung{x \cr}     & \myyoung{ a \cr} & \hspace{-0.5em}\bigg)
            & \arrow{y}{b} & \bigg( & \hspace{-0.5em} \myyoung{ x         & y \cr}           & \myyoung{ a & b \cr}           & \hspace{-0.5em}\bigg)
            \\[1ex]
            & \arrow{z}{b} & \bigg( & \hspace{-0.5em} \myyoung{x          & y                & z \cr}      & \myyoung{a       & b           & b \cr}              & \hspace{-0.5em}\bigg)
            & \arrow{w}{c} & \bigg( & \hspace{-0.5em} \myyoung{x \cr w    & y                & z \cr}      & \myyoung{c \cr a & b           & b \cr}              & \hspace{-0.5em}\bigg)
            \\[1ex]
            & \arrow{y}{c} & \bigg( & \hspace{-0.5em} \myyoung{ x         & z  \cr  w        & y           & y  \cr}          & \myyoung{ c & c  \cr  a           & b                   & b  \cr}   & \hspace{-0.5em}\bigg)
            & \arrow{w}{d} & \bigg( & \hspace{-0.5em} \myyoung{ z \cr x         & y  \cr  w        & w           & y                \cr}     & \myyoung{ d \cr c         & c  \cr  a           & b         & b \cr} & \hspace{-0.5em}\bigg)
            \\[2ex]
            & \arrow{y}{d} & \Bigg( & \hspace{-0.5em} \myyoung{ z  \cr  x & y  \cr  w        & w           & y                & y  \cr}     & \myyoung{ d  \cr  c & c  \cr  a           & b         & b & d  \cr} & \hspace{-0.5em}\Bigg)
            & \arrow{y}{d} & \Bigg( & \hspace{-0.5em} \myyoung{ z  \cr  x & y  \cr  w        & w           & y                & y           & y  \cr}             & \myyoung{ d  \cr  c & c  \cr  a & b & b       & d & d  \cr} & \hspace{-0.5em}\Bigg)
        \end{array}
    \end{equation*}
\end{example}

\begin{theorem}[Knuth] \label{th:oldRSK}
There is a one-to-one correspondence between
generalized permutations from $A$ to $B$
and pairs of tableaux $(P, Q)$ satisfying:
$P$ and $Q$ are of the same shape;
$P$ is semistandard with entries in $B$;
and $Q$ is semistandard with entries in $A$.
\end{theorem}

Two special cases of this correspondence coincide with the
correspondences discovered by Robinson~\cite{Robinson} and Schensted~\cite{Schensted}.
\begin{enumerate}
    \item
By encoding a permutation $\sigma$ of size $n$
as the generalized permutation
$\left(\begin{smallmatrix} 1 & 2 & \cdots & n \\ \sigma_1 & \sigma_2 & \cdots& \sigma_n \end{smallmatrix}\right)$,
where $\sigma_i = \sigma(i)$,
we obtain a bijection between permutations of size $n$ and pairs of standard
Young tableaux of the same shape of size $n$.
    \item
By encoding a word $w_1 \cdots w_k$ as
$\left(\begin{smallmatrix} 1 & 2 & \cdots & k \\ w_1 & w_2 & \cdots& w_k \end{smallmatrix}\right)$,
we obtain a bijection between words of length $k$ with entries in $[n]$
and pairs of Young tableaux of the same shape, the first semistandard with entries in $[n]$
and the second standard with entries in $[k]$.
\end{enumerate}

\section{RSK and representation theory}
\label{section.classical-RSK-rep-theory}

As described in Section~\ref{Introduction}, the RSK algorithm is
a combinatorial manifestation of certain direct-sum decompositions
from representation theory.
Below we consider the bijections induced by the RSK algorithm on permutations,
on words, and on generalized permutations. For each bijection, we describe
a vector space equipped with an action by two families of operators, and the
decomposition of the space into a direct sum of invariant subspaces.

\subsection{RSK and the group algebra $\CC S_n$}
\label{ssec:RSK-SnxSn}
When considered as a procedure on permutations, RSK establishes
a correspondence between permutations of $[n]$
and pairs of standard tableaux of the same shape of size $n$.
If $f^\lambda := \#\SYT(\lambda)$ denotes the number of standard tableaux of shape $\lambda$,
then this correspondence gives the enumerative result:
\begin{equation}
    \label{n-factorial-decomp}
    n! = \sum_{\lambda \vdash n} \left(f^\lambda\right)^2.
\end{equation}

This formula reflects the following classic decomposition result in
representation theory.
Let $\CC S_n$ denote the group algebra of the symmetric group $S_n$ with
coefficients in $\CC$. This is both a left and a right module over $S_n$,
and so it admits a decomposition into simple two-sided $S_n$-modules.
This decomposition takes the form
\begin{equation}
    \label{group-algebra-decomp}
    \CC S_n \cong
    \bigoplus_{\lambda \vdash n}
    \big(S^\lambda\big)^\ast
    \otimes
    S^\lambda,
\end{equation}
where $\{S^\lambda \mid \lambda \vdash n\}$ is a complete set of
non-isomorphic simple right $S_n$-modules,
and $(S^\lambda)^\ast = \Hom_{S_n}(S^\lambda, \CC)$.
Note that since $S^\lambda$ is a \emph{right} $S_n$-module,
its dual $\Hom_{S_n}(S^\lambda, \CC)$
is a \emph{left} $S_n$-module.
One recovers Equation~\eqref{n-factorial-decomp}
from the decomposition in Equation~\eqref{group-algebra-decomp}
by computing the dimension of the vector spaces
and noting that
$\dim(S^\lambda) = \dim((S^\lambda)^\ast) = f^\lambda$.

\subsection{RSK and the $GL(V) \times S_k$-structure on $V^{\otimes k}$}
\label{ssec:RSK-GLxSk}
When considered as a procedure on finite sequences, RSK establishes
a correspondence between finite sequences of length $k$ with entries in $[n]$
and pairs of tableaux $(P, Q)$ satisfying: $P$ and $Q$ have the same shape;
$P$ is a semistandard tableau with entries in $[n]$;
and $Q$ is a standard tableau with entries $[k]$.
Enumeratively,
\begin{equation}
    \label{nk-decomp}
    n^k = \sum_{\lambda \vdash k} \# \SSYT_{n}(\lambda) \,\cdot\, \# \SYT(\lambda),
\end{equation}
where $\SSYT_n(\lambda)$ is the set of semistandard tableaux of shape
$\lambda$ with entries in $[n]$, and where $\SYT(\lambda)$ is the set of standard tableaux of shape
$\lambda$.

This formula also reflects a classic decomposition result in representation
theory.
Let $V = \CC^n$ and consider the tensor product $V^{\otimes k}$ of $V$ with
itself $k$ times. This is a left $GL_n$-module, with $g \in GL_n$ acting
on simple tensors by
\begin{equation*}
    g \cdot \big(v_1 \otimes v_2 \otimes \cdots \otimes v_k\big)
    = g(v_1) \otimes g(v_2) \otimes \cdots \otimes g(v_k);
\end{equation*}
as well as a right $S_k$-module, with $\sigma \in S_k$ acting on simple tensors
by
\begin{equation*}
    \big(v_1 \otimes v_2 \otimes \cdots \otimes v_k\big) \cdot \sigma
    = v_{\sigma(1)} \otimes v_{\sigma(2)} \otimes \cdots \otimes v_{\sigma(k)};
\end{equation*}
and these two actions commute:
\begin{equation*}
    g \cdot \Big(\big(v_1 \otimes \cdots \otimes v_k\big) \cdot \sigma\Big)
    = \Big(g \cdot \big(v_1 \otimes \cdots \otimes v_k\big)\Big) \cdot \sigma.
\end{equation*}
Therefore, $V^{\otimes k}$ admits a decomposition into simple $GL_n \times
S_k$-modules, and this decomposition takes the form
\begin{equation*}
    V^{\otimes k} \cong
    \bigoplus_{\lambda \vdash k} W_n^\lambda \otimes S^\lambda,
\end{equation*}
where $W_n^\lambda$ is a simple left $GL_n$-module
and $S^\lambda$ is a simple right $S_k$-module.
Since $\dim W_n^\lambda = \#\SSYT_n(\lambda)$
and $\dim S^\lambda = \#\SYT(\lambda)$,
one immediately recovers Equation~\eqref{nk-decomp}.

The fact that the simple left $GL_n$-modules are also indexed by partitions
$\lambda$ is a special case of Schur--Weyl duality. It was first observed by
Schur \cite{Schur1, Schur2} in his thesis and later promoted by Weyl
\cite{Weyl} in his book on the representation theory of the classical groups.
The idea extends to a much more general result known as the {\it double
commutant theorem} (see, for instance, \cite{GoodWall}). Roughly speaking, it
states: if a vector space is acted on by two algebras of operators whose
actions mutually centralize each other, then the space decomposes into a direct
sum of tensors of two simple modules (for instance, $W_n^\lambda \otimes
S^\lambda$), and the two simple modules determine each other.

\subsection{RSK and the $GL_n \times GL_k$-structure on $\CC[\{x_{ij}\}]$}
\label{ssec:RSK-GLxGL}
Knuth's generalization establishes a correspondence between monomials of degree
$r$ in $n$ sets of $k$ variables $\{ x_{ij} \mid i \in [n], j \in [k] \}$ and
pairs of semistandard tableaux $(P, Q)$ of size $r$ satisfying:
$P$ and $Q$ have the same shape;
$P$ is semistandard with entries in $[n]$;
$Q$ is semistandard with entries in $[k]$.
This is achieved by encoding monomials as
generalized permutations from $[k]$ to $[n]$:
each occurrence of $x_{ij}$ is encoded by
the column $\left(\begin{smallmatrix} j \\ i \end{smallmatrix}\right)$,
for example $x_{12}^3 x_{14} x_{23}^2$ becomes
$\left(\begin{smallmatrix} 2 & 2 & 2 & 3 & 3 & 4 \\ 1 & 1 & 1 & 2 & 2 & 1 \end{smallmatrix}\right)$.

This correspondence reflects a decomposition of the polynomial ring
generated by the variables $\{ x_{ij} \mid i \in [n], j \in [k] \}$
when it is viewed as a $GL_n \times GL_k$-module.
To define the module structure, we first arrange the variables in the form of
an $n \times k$ matrix:
\begin{equation*}
    X =
    \begin{bmatrix}
        x_{11} & x_{12} & \ldots & x_{1k} \\
        x_{21} & x_{22} & \ldots & x_{2k} \\
        \vdots & \vdots & \ddots & \vdots \\
        x_{n1} & x_{n2} & \ldots & x_{nk}
    \end{bmatrix}.
\end{equation*}
The left action of $A \in GL_n$ on the polynomial ring corresponds to multiplying $X$ on
the left by $A$. Explicitly, the variable $x_{ij}$ is replaced with
$(AX)_{ij} = \sum_{l} a_{il} x_{lj}$.
The right action of $B \in GL_k$ corresponds to
multiplying $X$ on the right by $B$ (explicitly, $x_{ij} \mapsto
(XB)_{ij}$). Since these actions correspond to multipying
$X$ on the left and right by matrices, the fact that the two actions commute
is a consequence of the associativity of matrix multiplication.

Consequently, the polynomial ring admits a direct sum decomposition
whose summands are tensors of pairs of simple modules
(see, for instance,~\cite{Howe}):
\begin{equation*}
    \CC[X] \cong
    \bigoplus_{\lambda} W_n^\lambda \otimes \big(W_k^\lambda\big)^\ast,
\end{equation*}
where $\lambda$ runs over all partitions,
$W_n^\lambda$ is the simple left $GL_n$-module indexed by $\lambda$, and
$(W_k^\lambda)^\ast$ is the simple right $GL_k$-module indexed by $\lambda$.
Since the dimension of $W_m^\lambda$ is the number of semistandard
tableaux of shape $\lambda$ with entries in $[m]$,
one sees that the set of monomials in $\{ x_{ij} \mid i \in [n], j \in [k] \}$
is in bijection with the set of pairs $(P, Q)$ of semistandard tableaux of the
same shape with $P$ having entries in $[k]$ and $Q$ having entries in $[n]$.

In particular, for every multiset
$\dcl 1^{\alpha_1}, \ldots, k^{\alpha_k} \dcr$,
the monomials $\prod_{i=1}^n \prod_{j=1}^k x_{ij}^{b_{ij}}$
satisfying $\sum_{i=1}^n b_{ij} = \alpha_j$ for all $j \in [k]$
span a $GL_n$-invariant subspace of $\CC[X]$ whose dimension is equal to
$\prod_{i=1}^{k} \binom{n+\alpha_i-1}{\alpha_i}$. These monomials
are in bijection with the pairs $(P, Q)$ of semistandard tableaux of the
same shape with $P$ having entries in $[n]$ and $Q$ having content
$\dcl 1^{\alpha_1}, \ldots, k^{\alpha_k} \dcr$.
From this we immediately obtain
\begin{equation}
    \label{eq:monomials-as-ssyt-ssyt}
    \prod_{i=1}^{k}
    \binom{n+\alpha_i-1}{\alpha_i}
    = \sum_{\lambda \vdash r} \sum_{S}
    \# \SSYT(\lambda, S) \cdot \#\SSYT\Big(\lambda, \dcl 1^{\alpha_1}, 2^{\alpha_2}, \ldots, k^{\alpha_k} \dcr\Big)~,
\end{equation}
where the inner sum is over multisets $S$ of size $r = \sum_{i=1}^k \alpha_i$ with entries in $[n]$.

\section{Application: a new insertion on generalized permutations} \label{sec:app1}

\begin{center}
    \vskip 2mm
    \framebox{\begin{minipage}{0.81\textwidth}
        Throughout this section, semistandard multiset tableaux are defined
        using \textbf{graded lexicographic order};
        see Sections~\ref{sec:multiset-orders}~and~\ref{sec:multiset-tableaux}
        for details.
    \end{minipage}}
    \vskip 4mm
\end{center}

Section~\ref{section.classical-RSK-rep-theory} illustrated how RSK
parallels the direct-sum decomposition of three distinct vector
spaces. In each setting, the correspondence was facilitated by parameterizing
a basis of the vector space by generalized permutations. More specifically,
permutations, then words, and finally monomials
were encoded as generalized permutations to which the RSK insertion procedure
was applied (\cf~Section~\ref{ssec:RSK-SnxSn}--Section~\ref{ssec:RSK-GLxGL}).
In this section, we begin with an alternative encoding, which produces new
combinatorial and enumerative results that parallel a different decomposition
of the polynomial ring of Section~\ref{ssec:RSK-GLxGL}.

\subsection{The correspondence}
Let $\left(\begin{smallmatrix} a_1 & a_2 & \cdots & a_r \\ b_1 & b_2 & \cdots& b_r \end{smallmatrix}\right)$
be a generalized permutation from $[k]$ to $[n]$.
We transform this into a generalized permutation from \defn{multisets over $[k]$ to $[n]$} as follows.
The columns of the generalized permutation are $\binom{M_i}{i}$, where $M_i
= \dcl a_j \mid b_j = i \dcr$ and $i \in [n]$. The columns are ordered so that
the entries of the top row are weakly-increasing in graded lexicographic order
and $i < i'$ whenever $M_i = M_{i'}$.

This encoding, together with Theorem~\ref{th:oldRSK}, establishes the following result.
\begin{prop}\label{prop:twolinenewRSK}
There is a one-to-one correspondence between generalized permutations
$\left(\begin{smallmatrix} a_1 & a_2 & \cdots & a_n \\ b_1 & b_2 & \cdots& b_n \end{smallmatrix}\right)$,
where $a_1, \ldots, a_n$ are multisets over $[k]$
and $b_1, \ldots, b_n$ are distinct elements of $[n]$,
and pairs $(S, T)$ satisfying:
$S$ and $T$ are tableaux of the same shape;
$S$ is a standard Young tableau of size $n$; and
$T$ is a semistandard multiset tableau of content
$a_1 \uplus \cdots \uplus a_n$.
\end{prop}

\begin{example}
    Consider the generalized permutation from $[6]$ to $[5]$,
    \[
        \left(\!\!\begin{array}{ccccccccccccc}
        1&1&1&2&2&3&3&3&3&4&6&6&6\\
        1&5&5&2&3&1&3&5&5&1&1&2&3
        \end{array}\!\!\right)~,
    \]
    with which we associate the following generalized permutation whose top row
    consists of (possibly empty) multisets over $[6]$ and whose bottom row
    are the elements $\{1, 2, 3, 4, 5\}$, each appearing exactly
    once:
    \[
        \left(\!\!\begin{array}{ccccc}
        \dcl\dcr&\dcl2,6\dcr&\dcl2,3,6\dcr&\dcl1,1,3,3\dcr&\dcl1,3,4,6\dcr\cr
        4&2&3&5&1
        \end{array}\!\!\right)~.
    \]
    This generalized permutation in turn corresponds to the following pair of
    tableaux:
    \squaresize=17pt
    $$\left(\!\!\raisebox{-20pt}{\young{4\cr2\cr1&3&5\cr}~, \young{\hbox{\tiny 1346}\cr\hbox{\tiny 26}\cr&\hbox{\tiny 236}&\hbox{\tiny 1133}\cr}}\right)~.$$
\end{example}

\begin{example}
Consider the case where $n=3$, $k=2$ and the
generalized permutations are of the form
$\left(\begin{smallmatrix} 1 & 1 & 2 \\ a & b & c \end{smallmatrix}\right)$.
The pairs of tableaux under the correspondence of
Proposition~\ref{prop:twolinenewRSK} are depicted in
Figure~\ref{fig:twolinenewRSK-example},
whereas the pairs of tableaux under the
usual RSK correspondence are depicted in
Figure~\ref{fig:twolineoldRSK-example}.

\begin{figure}[ht!]
\squaresize=13pt
\def\genperm#1#2{$\left(\begin{matrix} 1 & 1 & 2 \\ #1 & #2 & c \end{matrix}\right)$}
\def\myyoung#1{\begingroup\setbox0=\hbox{\young{#1}}\parbox{\wd0}{\box0}\endgroup}
\begin{tabular}{c|r@{\hskip 0.5\tabcolsep}lcr@{\hskip 0.5\tabcolsep}lcr@{\hskip 0.5\tabcolsep}l}
                    & \multicolumn{2}{c}{$c=1$} && \multicolumn{2}{c}{$c=2$} && \multicolumn{2}{c}{$c=3$} \\ \toprule
    \genperm{1}{1}  & $\myyoung{2\cr1&3\cr}$      & $\myyoung{\hbox{\tiny 112}\cr&\cr}$  && $\myyoung{3\cr2\cr1\cr}$   & $\myyoung{11\cr2\cr\cr}$            && $\myyoung{2\cr1&3\cr}$  & $\myyoung{11\cr&2\cr}$\\[3ex]
    \genperm{1}{2}  & $\myyoung{3\cr2\cr1\cr}$    & $\myyoung{12\cr1\cr\cr}$             && $\myyoung{3\cr1&2\cr}$     & $\myyoung{1\cr&12\cr}$              && $\myyoung{1&2&3\cr}$    & $\myyoung{1&1&2\cr}$\\[4ex]
    \genperm{1}{3}  & $\myyoung{2\cr1&3\cr}$      & $\myyoung{12\cr&1\cr}$               && $\myyoung{3\cr1&2\cr}$     & $\myyoung{2\cr1&1\cr}$              && $\myyoung{2\cr1&3\cr}$  & $\myyoung{1\cr&12\cr}$\\[4ex]
    \genperm{2}{2}  & $\myyoung{3\cr1&2\cr}$      & $\myyoung{2\cr&11\cr}$               && $\myyoung{3\cr1&2\cr}$     & $\myyoung{\hbox{\tiny 112}\cr&\cr}$ && $\myyoung{3\cr1&2\cr}$  & $\myyoung{11\cr&2\cr}$\\[4ex]
    \genperm{2}{3}  & $\myyoung{2\cr1&3\cr}$      & $\myyoung{2\cr1&1\cr}$               && $\myyoung{3\cr1&2\cr}$     & $\myyoung{12\cr&1\cr}$              && $\myyoung{1&2&3\cr}$    & $\myyoung{&1&12\cr}$\\[4ex]
    \genperm{3}{3}  & $\myyoung{2\cr1&3\cr}$      & $\myyoung{2\cr&11\cr}$               && $\myyoung{1&2&3\cr}$       & $\myyoung{&2&11\cr}$                && $\myyoung{1&2&3\cr}$    & $\myyoung{&&\hbox{\tiny 112}\cr}$
\end{tabular}
\caption{The associated pair of tableaux by the correspondence of Proposition~\ref{prop:twolinenewRSK}.}
    \label{fig:twolinenewRSK-example}
\end{figure}

\begin{figure}[ht!]
\squaresize=13pt
\def\genperm#1#2{$\left(\begin{matrix} 1 & 1 & 2 \\ #1 & #2 & c \end{matrix}\right)$}
\def\myyoung#1{\begingroup\setbox0=\hbox{\young{#1}}\parbox{\wd0}{\box0}\endgroup}
\begin{tabular}{c|r@{\hskip 0.5\tabcolsep}lcr@{\hskip 0.5\tabcolsep}lcr@{\hskip 0.5\tabcolsep}l}
    & \multicolumn{2}{c}{$c=1$} && \multicolumn{2}{c}{$c=2$} && \multicolumn{2}{c}{$c=3$} \\ \toprule
    \genperm{1}{1}  & $\myyoung{1&1&1\cr}$ & $\myyoung{1&1&2\cr}$&& $\myyoung{1&1&2\cr}$ & $\myyoung{1&1&2\cr}$&& $\myyoung{1&1&3\cr}$ & $\myyoung{1&1&2\cr}$\\[3ex]
    \genperm{1}{2}  & $\myyoung{2\cr1&1\cr}$ & $\myyoung{2\cr1&1\cr}$&& $\myyoung{1&2&2\cr}$ & $\myyoung{1&1&2\cr}$&& $\myyoung{1&2&3\cr}$ & $\myyoung{1&1&2\cr}$\\[3ex]
    \genperm{1}{3}  & $\myyoung{3\cr1&1\cr}$ & $\myyoung{2\cr1&1\cr}$&& $\myyoung{3\cr1&2\cr}$ & $\myyoung{2\cr1&1\cr}$&& $\myyoung{1&3&3\cr}$ & $\myyoung{1&1&2\cr}$\\[3ex]
    \genperm{2}{2}  & $\myyoung{2\cr1&2\cr}$ & $\myyoung{2\cr1&1\cr}$&& $\myyoung{2&2&2\cr}$ & $\myyoung{1&1&2\cr}$&& $\myyoung{2&2&3\cr}$ & $\myyoung{1&1&2\cr}$\\[3ex]
    \genperm{2}{3}  & $\myyoung{2\cr1&3\cr}$ & $\myyoung{2\cr1&1\cr}$&& $\myyoung{3\cr2&2\cr}$ & $\myyoung{2\cr1&1\cr}$&& $\myyoung{2&3&3\cr}$ & $\myyoung{1&1&2\cr}$\\[3ex]
    \genperm{3}{3}  & $\myyoung{3\cr1&3\cr}$ & $\myyoung{2\cr1&1\cr}$&& $\myyoung{3\cr2&3\cr}$ & $\myyoung{2\cr1&1\cr}$&& $\myyoung{3&3&3\cr}$ & $\myyoung{1&1&2\cr}$
\end{tabular}
\caption{The associated pair of tableaux by the usual RSK correspondence.}
    \label{fig:twolineoldRSK-example}
\end{figure}
\end{example}

\subsection{Special case: insertion on words}
In the special case where the top row of
$\left(\begin{smallmatrix} a_1 & a_2 & \cdots & a_r \\ b_1 & b_2 & \cdots& b_r \end{smallmatrix}\right)$
satisfies $a_j = j$ for all $j$, we obtain the following correspondence.

\begin{cor} \label{prop:enumnpk}
    There is a one-to-one correspondence between words of length $k$ with
    entries in $[n]$ and pairs $(S, T)$ satisfying:
    $S$ and $T$ are tableaux of the same shape;
    $S$ is a standard Young tableau of size $n$;
    $T$ is a standard multiset tableau of content $[k]$.
\end{cor}

\begin{example}
    The generalized permutation associated with the word
    $155231315$ over the alphabet $[6]$ is
    \[
        \left( \begin{array}{cccccc} \{ \}& \{ \} & \{4\} & \{5,7\} &\{1,6,8\} &\{2,3,9\} \\
        4 & 6 & 2 & 3 & 1&5\end{array} \right),
    \]
    which corresponds to the pair
    \begin{equation*}
    \squaresize=17pt
    \left(\!\!\raisebox{-20pt}{\young{4\cr2&6\cr1&3&5\cr}~,
    \young{\hbox{\tiny 168}\cr\hbox{\tiny 4}&\hbox{\tiny 57}\cr\hbox{}&\hbox{}&\hbox{\tiny 239}\cr}}\right)~.
    \end{equation*}
\end{example}

\begin{example}
Consider the (relatively small) example where $n=4$ and $k=2$ so that there
are $16$ words of length $2$ with entries in $\{1,2,3,4\}$.
Figure~\ref{fig:oldRSK-ab-example} depicts the
pairs of tableaux associated with these words by the usual RSK algorithm;
and
Figure~\ref{fig:newRSK-ab-example} depicts the
pairs of tableaux associated with these words by the correspondence
of Corollary~\ref{prop:enumnpk}.
\begin{figure}[ht!]
\squaresize=13pt
\def\genperm#1#2{$\left(\begin{matrix} 1 & 1 & 2 \\ #1 & #2 & c \end{matrix}\right)$}
\def\myyoung#1{\begingroup\setbox0=\hbox{\young{#1}}\parbox{\wd0}{\box0}\endgroup}
\begin{tabular}{c| r@{\hskip 0.5\tabcolsep}lc r@{\hskip 0.5\tabcolsep}lc r@{\hskip 0.5\tabcolsep}lc r@{\hskip 0.5\tabcolsep}l}
                    & \multicolumn{2}{c}{$b=1$} && \multicolumn{2}{c}{$b=2$} && \multicolumn{2}{c}{$b=3$} && \multicolumn{2}{c}{$b=4$} \\ \toprule
    $a = 1$ &
    $\myyoung{1&1\cr}$ & $\myyoung{1&2\cr}$&&
    $\myyoung{1&2\cr}$ & $\myyoung{1&2\cr}$&&
    $\myyoung{1&3\cr}$ & $\myyoung{1&2\cr}$&&
    $\myyoung{1&4\cr}$ & $\myyoung{1&2\cr}$\\[5pt]
    $a = 2$            &
    $\myyoung{2\cr1\cr}$ & $\myyoung{2\cr1\cr}$&&
    $\myyoung{2&2\cr}$ & $\myyoung{1&2\cr}$&&
    $\myyoung{2&3\cr}$ & $\myyoung{1&2\cr}$&&
    $\myyoung{2&4\cr}$ & $\myyoung{1&2\cr}$\\[10pt]
    $a = 3$            &
    $\myyoung{3\cr1\cr}$ & $\myyoung{2\cr1\cr}$&&
    $\myyoung{3\cr2\cr}$ & $\myyoung{2\cr1\cr}$&&
    $\myyoung{3&3\cr}$ & $\myyoung{1&2\cr}$&&
    $\myyoung{3&4\cr}$ & $\myyoung{1&2\cr}$\\[10pt]
    $a = 4$            &
    $\myyoung{4\cr1\cr}$ & $\myyoung{2\cr1\cr}$&&
    $\myyoung{4\cr2\cr}$ & $\myyoung{2\cr1\cr}$&&
    $\myyoung{4\cr1\cr}$ & $\myyoung{2\cr1\cr}$&&
    $\myyoung{4&4\cr}$ & $\myyoung{1&2\cr}$
\end{tabular}
\caption{The pair of tableaux associated with the word $ab$ by the usual RSK correspondence.}
\label{fig:oldRSK-ab-example}
\end{figure}
\begin{figure}[ht!]
\squaresize=10pt
\def\genperm#1#2{$\left(\begin{matrix} 1 & 1 & 2 \\ #1 & #2 & c \end{matrix}\right)$}
\def\myyoung#1{\begingroup\setbox0=\hbox{\small\young{#1}}\parbox{\wd0}{\box0}\endgroup}
\begin{tabular}{c| r@{\hskip 0.5\tabcolsep}lc r@{\hskip 0.5\tabcolsep}lc r@{\hskip 0.5\tabcolsep}lc r@{\hskip 0.5\tabcolsep}l}
            & \multicolumn{2}{c}{$b=1$} && \multicolumn{2}{c}{$b=2$} && \multicolumn{2}{c}{$b=3$} && \multicolumn{2}{c}{$b=4$} \\ \toprule
    $a = 1$ & $\myyoung{2\cr1&3&4\cr}$ & $\myyoung{12\cr&&\cr}$&& $\myyoung{3&4\cr1&2\cr}$ & $\myyoung{1&2\cr&\cr}$&& $\myyoung{2&4\cr1&3\cr}$ & $\myyoung{1&2\cr&\cr}$&& $\myyoung{2\cr1&3&4\cr}$ & $\myyoung{1\cr&&2\cr}$\\[10pt]
    $a = 2$ & $\myyoung{3\cr2\cr1&4\cr}$ & $\myyoung{2\cr1\cr&\cr}$&& $\myyoung{3\cr1&2&4\cr}$ & $\myyoung{12\cr&&\cr}$&& $\myyoung{4\cr1&2&3\cr}$ & $\myyoung{1\cr&&2\cr}$&& $\myyoung{3\cr1&2&4\cr}$ & $\myyoung{1\cr&&2\cr}$\\[15pt]
    $a = 3$ & $\myyoung{4\cr2\cr1&3\cr}$ & $\myyoung{2\cr1\cr&\cr}$&& $\myyoung{4\cr3\cr1&2\cr}$ & $\myyoung{2\cr1\cr&\cr}$&& $\myyoung{4\cr1&2&3\cr}$ & $\myyoung{12\cr&&\cr}$&& $\myyoung{1&2&3&4\cr}$ & $\myyoung{&&1&2\cr}$\\[15pt]
    $a = 4$ & $\myyoung{2\cr1&3&4\cr}$ & $\myyoung{2\cr&&1\cr}$&& $\myyoung{3\cr1&2&4\cr}$ & $\myyoung{2\cr&&1\cr}$&& $\myyoung{4\cr1&2&3\cr}$ & $\myyoung{2\cr&&1\cr}$&& $\myyoung{1&2&3&4\cr}$ & $\myyoung{&&&12\cr}$
\end{tabular}
\caption{The pair of tableaux associated with the word $ab$ by the correspondence of Corollary~\ref{prop:enumnpk}.}
\label{fig:newRSK-ab-example}
\end{figure}
\end{example}

\subsection{Enumerative results}

By restricting the correspondence of Proposition~\ref{prop:twolinenewRSK}
to generalized permutations
$\left(\begin{smallmatrix} a_1 & a_2 & \cdots & a_n \\ b_1 & b_2 & \cdots& b_n \end{smallmatrix}\right)$
satisfying $a_1 \uplus \cdots \uplus a_n = \dcl 1^{\alpha_1}, 2^{\alpha_2}, \ldots, k^{\alpha_k} \dcr$,
we obtain the following enumerative statement:
\begin{equation*}
    \prod_{i=1}^{k}
    \binom{n+\alpha_i-1}{\alpha_i}
    = \sum_{\lambda \vdash n} \#\SYT(\lambda) \cdot
    \# \SSMT\Big(\lambda, \dcl 1^{\alpha_1}, 2^{\alpha_2}, \ldots, k^{\alpha_k} \dcr\Big)~.
\end{equation*}
Compare this with Equation~\eqref{eq:monomials-as-ssyt-ssyt},
which is the enumerative statement that accompanies Theorem~\ref{th:oldRSK}.

The analogous enumerative statement that accompanies
Corollary~\ref{prop:enumnpk} is the special case where
$\alpha_i = 1$ for all $i$:
\begin{equation}
    \label{eq:n^k-as-syt-ssmt}
    n^k = \sum_{\lambda \vdash n} \# \SYT(\lambda) \cdot
    \#\SMT(\lambda,k)~.
\end{equation}
Compare this with the enumerative statement obtained from the usual application
of the RSK correspondence to words (see Section~\ref{ssec:RSK-GLxSk} and
Equation~\eqref{nk-decomp}):
\begin{equation}
    \label{eq:n^k-as-ssyt-syt}
    n^k = \sum_{\mu \vdash k} \sum_{C} \#\SSYT(\mu, C) \cdot
    \#\SYT(\mu),
\end{equation}
where the inner sum is over
all multisets $C$ of size $k$ with entries in $\{1,2,\ldots,n\}$.

\subsection{Connections with representation theory}
\label{ssec:reptheoryconnections}

Consider the vector space $V^{\otimes k}$, where $V = \CC^n$.
This space admits an action of the symmetric group $S_n$
as well as an action of the partition algebra $P_k(n)$.
The $S_n$-action is obtained by identifying the symmetric group $S_n$ with the
subgroup of $GL_n$ consisting of the permutation matrices and restricting the
$GL_n$-action on $V^{\otimes k}$ defined in Section~\ref{ssec:RSK-GLxSk}.
The precise definition of the $P_k(n)$-action is not necessary here,
so we refer the interested reader to~\cite[Eq. (1.3.4)]{Halverson}.
More information on the partition algebra $P_k(n)$ is presented in
Section~\ref{sec:diagramalgebras}.

These two actions are closely related when $n \geqslant 2k$:
the algebra of linear endomorphisms of $V^{\otimes k}$ that commute with the
$S_n$-action is isomorphic to $P_k(n)$; and conversely, the algebra of linear
endomorphisms of $V^{\otimes k}$ that
commute with the $P_k(n)$-action is isomorphic to $\CC S_n$ \cite{Jones}.
That is, provided that $n \geqslant 2k$,
\begin{equation*}
    \End_{S_n}\big(V^{\otimes k}\big) \cong P_k(n)
    \qquad\text{and}\qquad
    \End_{P_k(n)}\big(V^{\otimes k}\big) \cong \CC S_n.
\end{equation*}
Consequently, $V^{\otimes k}$ admits a direct sum decomposition
whose summands are tensors of a simple $S_n$-module
and a simple $P_k(n)$-module
(see, for instance,~\cite[Theorem 3.2.2]{HalRam} or \cite[Theorem~8.3.18]{CST}):
\begin{equation*}
    V^{\otimes k}
    \cong
    \bigoplus_{\substack{\lambda \vdash n\\|\o\lambda|\leqslant k}} S^\lambda \otimes V_{P_k(n)}^{\o\lambda},
\end{equation*}
where $S^\lambda$ is the simple $S_n$-module indexed by $\lambda$
and $V_{P_k(n)}^{\o\lambda}$ is the simple $P_k(n)$-module indexed by $\o\lambda$
(recall from Section~\ref{section.partitions} that $\o\lambda$ is obtained from $\lambda$ by deleting the first row).
Since the dimension of $S^\lambda$ is the number of standard tableaux of shape $\lambda$
and the dimension of $V_{P_k(n)}^{\o\lambda}$ is the number of standard
multiset tableaux of shape $\lambda$ and content~$[k]$, we
immediately obtain the enumerative result in Equation~\eqref{eq:n^k-as-syt-ssmt}.

As pointed out in the introduction, this combinatorial interpretation for
the dimension is different from that of~\cite{HalLew, MarRol}, which makes use
of vacillating tableaux instead of multiset tableaux. It is more closely
aligned with the results in~\cite{BH,BHH,OZ,HJ}.

\section{Application: Diagram Algebras}
\label{sec:diagramalgebras}

\begin{center}
    \vskip 2mm
    \framebox{\begin{minipage}{0.75\textwidth}
        Throughout this section, standard multiset tableaux are defined
        using \textbf{last letter order};
        see Sections~\ref{sec:multiset-orders}~and~\ref{sec:multiset-tableaux}
        for details.
    \end{minipage}}
    \vskip 4mm
\end{center}

For any parameter $n$ and positive integer $k$,  the partition algebra, $P_k(n)$, is defined as the complex vector space 
with basis given by the set partitions on  two disjoint sets  
$[k] \cup [\overline{k}]=\{1, 2, \ldots, k\} \cup \{\o1,\o2, \ldots, \overline{k}\}$:
\[
    P_k(n) =\mbox{span}_{\mathbb{C}} \{\pi \mid \pi \vdash [k]\cup[\o{k}]\}.
\]
Although we do not define the product here, as we will not use it explicitly,
we remark that the dependency of the algebra on $n$
arises when we multiply the set partitions \cite{HalRam}.

A \defn{diagram} is a graphical representation of a set partition of the set
$[k] \cup [\overline{k}]$: the vertex set of the graph is $[k] \cup [\overline{k}]$  arranged in two horizontal 
rows, where the top row is labelled by $1,2, \ldots, k$ and the bottom row are labelled 
by $\o1, \o2, \ldots, \overline{k}$; and there is a path connecting two vertices if and only if
they belong to the same block of the set partition. Note that there is more than one graph that
represents a set partition, but this is immaterial to the following.
In our examples, we will connect vertices in the same block with
a cycle.

\begin{example}
    \label{ex:diagram}
The set partition
$\pi = \{\{ 1,2,4,\o2,\o5\},\{3\}, \{5,6,7, \o3,\o4,\o6,\o7\},\{8,\o8\}, \{\o1\}\}$ 
is represented by the following diagram:
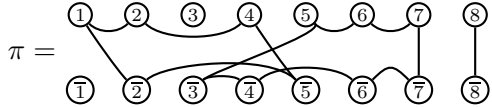
\begin{center}$\pi = $
\begin{tikzpicture}[scale = 0.5,thick, baseline={(0,-1ex/2)}]
\tikzstyle{vertex} = [shape = circle, minimum size = 7pt, inner sep = 1pt]
\node[vertex] (G--5) at (6.0, -1) [shape = circle, draw] {\tiny $\overline 5$};
\node[vertex] (G--4) at (4.5, -1) [shape = circle, draw] {\tiny $\overline 4$};
\node[vertex] (G--3) at (3.0, -1) [shape = circle, draw] {\tiny $\overline 3$};
\node[vertex] (G-5) at (6.0, 1) [shape = circle, draw] {\tiny 5};
\node[vertex] (G-4) at (4.5, 1) [shape = circle, draw] {\tiny 4};
\node[vertex] (G--2) at (1.5, -1) [shape = circle, draw] {\tiny $\overline 2$};
\node[vertex] (G--1) at (0.0, -1) [shape = circle, draw] {\tiny $\overline 1$};
\node[vertex] (G-1) at (0.0, 1) [shape = circle, draw] {\tiny 1};
\node[vertex] (G-2) at (1.5, 1) [shape = circle, draw] {\tiny 2};
\node[vertex] (G-3) at (3.0, 1) [shape = circle, draw] {\tiny 3};
\node[vertex] (G-6) at (7.5, 1) [shape = circle, draw] {\tiny 6};
\node[vertex] (G--6) at (7.5, -1) [shape = circle, draw] {\tiny $\overline 6$};
\node[vertex] (G-7) at (9.0, 1) [shape = circle, draw] {\tiny 7};
\node[vertex] (G--7) at (9.0, -1) [shape = circle, draw] {\tiny $\overline 7$};
\node[vertex] (G-8) at (10.5, 1) [shape = circle, draw] {\tiny 8};
\node[vertex] (G--8) at (10.5, -1) [shape = circle, draw] {\tiny $\overline 8$};
\draw (G-1) .. controls +(0.5, -0.5) and +(-0.5, -0.5) .. (G-2);
\draw (G-2) .. controls +(0.7, -0.7) and +(-0.7, -0.7) .. (G-4);
\draw (G--5) .. controls +(-1, 1) and +(0.7, 0.7) .. (G--2);
\draw (G-1) .. controls +(0.3, -0.5) and +(-.5, .5) .. (G--2);
\draw (G-4) .. controls +(0.3, -0.5) and +(-.5, .5) .. (G--5);
\draw (G-5) .. controls +(0.5, -0.5) and +(-0.5, -0.5) .. (G-6);
\draw (G-6) .. controls +(0.5, -0.5) and +(-0.5, -0.5) .. (G-7);
\draw (G--3) .. controls +(0.4, 0.4) and +(-0.4, 0.4) .. (G--4);
\draw (G--4) .. controls +(0.7, 0.7) and +(-0.7, 0.7) .. (G--6);
\draw (G-5) .. controls +(0.5, -0.5) and +(0.5, 0.5) .. (G--3);
\draw (G--6) .. controls +(0.7, 0.7) and +(-0.7, 0.7) .. (G--7);
\draw (G-7)--(G--7);
\draw (G-8)--(G--8);
\end{tikzpicture}
\end{center}
\vskip 2mm
\end{example}

The partition algebra $P_k(n)$ is semisimple whenever the parameter $n\notin \{0, 1, \ldots, 2k-2\}$ in which
case the irreducible representations are indexed by partitions $\lambda$ with $0\leqslant |\lambda| \leqslant k$ \cite{MS}.
We assume throughout that $n\geqslant 2k$, so that $P_k(n)$ is semisimple and
isomorphic to $\End_{S_n}(V^{\otimes k})$ as explained in Section~\ref{ssec:reptheoryconnections}.

In \cite{HalLew,MarRol}, the authors introduce RSK-type algorithms
between partition algebra diagrams and pairs of
paths in the Bratteli diagram of the partition algebras;
in \cite{HalLew} these paths are called vacillating tableaux.
In \cite{BH}, the authors define a bijection between vacillating tableaux and standard multiset tableaux.
In this section we provide a different bijection from partition
algebra diagrams to standard multiset tableaux.  This algorithm not
only encodes the representation theory of the partition algebra, in the sense that the tableaux of 
shape $\lambda$  index an irreducible representation associated with $\lambda$,
but it also encodes the representation theory of subalgebras of the partition
algebra when we restrict the set of diagrams considered.
This allows us to obtain enumerative results for representations of various diagram algebras
using standard multiset tableaux.

In this section, we only consider centralizer algebras acting on $V^{\otimes k}$, but this construction
indicates that more general diagram algebras are also of interest.
If instead one considers the centralizer algebras acting on the polynomial ring $\CC[X]$
(as described in Section \ref{ssec:RSK-GLxGL}), then the corresponding diagrams would
have repeated entries and the dimensions of the irreducible representations will be
subsets of semistandard multiset tableaux.  Currently little is known about these centralizer
algebras, see for instance \cite{NPS, OZ2}.

\subsection{The correspondence}
\label{ssec:RSK-for-diagrams}

A block in a set partition $\pi$ is called \defn{propagating} if it contains vertices in both $[k]$ and $[\overline{k}]$.
For example, $\{1, 2, 4, \o2, \o5\}$ is a propagating block.  A block is called \defn{non-propagating} otherwise.
The number of propagating blocks in $\pi$ is called the \defn{propagating number}.
We denote the propagating number
by $\operatorname{pr}(\pi)$. For example,  the set partition $\pi$ in Example~\ref{ex:diagram} has 
$\operatorname{pr}(\pi)=3$.

Let $\pi = \{\pi_1, \pi_2, \ldots, \pi_r\}$ be a set partition of $[k]\cup [\o{k}]$.
We associate with $\pi$ a pair $(T, S)$ of standard multiset tableaux as follows.
To begin,
\begin{itemize}
    \item
        let $\pi_{j_1},\pi_{j_2}, \ldots, \pi_{j_p}$ denote the propagating
        blocks of $\pi$ ordered so that $\pi_{j_1}^+ < \cdots < \pi_{j_p}^+$ in
        the last letter order, where $\pi_{j}^+ = \pi_{j} \cap [k]$;

    \item
        let $\sigma_{i_1}, \ldots, \sigma_{i_a} \subseteq [k]$ denote
        the non-propagating blocks contained in $[k]$ and ordered so that
        $\sigma_{i_1} < \cdots < \sigma_{i_a}$ in the last letter order;

    \item
        let $\tau_{i_1}, \ldots, \tau_{i_b} \subseteq [\o{k}]$ denote
        the non-propagating blocks contained in $[\o{k}]$ and ordered so that
        $\tau_{i_1} < \cdots < \tau_{i_b}$ in the last letter order.
\end{itemize}
Let $(P, Q)$ denote the pair of standard multiset tableaux obtained by
applying the RSK algorithm to the generalized permutation
\begin{equation*}
    \left(
        \begin{array}{cccc}
            \pi_{j_1}^+ &
            \pi_{j_2}^+ &
            \cdots &
            \pi_{j_p}^+
            \\[4pt]
            \pi_{j_1}^- &
            \pi_{j_2}^- &
            \cdots &
            \pi_{j_p}^-
        \end{array}
    \right),
\end{equation*}
where $\pi_{j}^+ = \pi_j \cap [k]$ and $\pi_{j}^- = \pi_j \cap [\o{k}]$.
Let $T$ be the tableau obtained from $P$ by adjoining a row containing $n-p-b$
empty cells followed by cells labelled $\tau_{i_1}, \ldots, \tau_{i_b}$.
Let $S$ be the tableau obtained from $Q$ by adjoining a row containing $n-p-a$
empty cells followed by cells labelled $\sigma_{i_1}, \ldots, \sigma_{i_a}$.

\begin{example}
Let $\pi = \{\{ 2,3,4,\o4,\o5\},\{5,\o2,\o3\},\{1,6,\o7,\o8\},\{7,8\},\{9,\o6\},\{\o1\},\{\o9\}\}\in P_9(18)$.
The non-propagating blocks are $\{\o1\}$, $\{\o9\}$ and $\{7,8\}$, and
the generalized permutation constructed from the propagating blocks is
\begin{equation*}
    \left( \begin{array}{cccc}  \{2,3,4\} & \{5\} & \{1,6\} & \{9\} \\[2pt]
\{\o4,\o5\} & \{\o2,\o3\} & \{\o7,\o8\} & \{\o6\}\end{array}\right).
\end{equation*}
Apply the RSK algorithm to obtain the following pair of multiset tableaux:
\begin{equation*}
    \squaresize=14pt
    P =
    \begin{array}{c}
        \scriptsize
        \young{\o4\o5 & \o7\o8\cr \o2 \o3& \o6 \cr}
    \end{array},
    \qquad
    Q =
    \begin{array}{c}
        \scriptsize
        \young{5& 9\cr 234 &16\cr}
    \end{array}.
\end{equation*}
Finally, adjoin a new row to $P$ and a new row to $Q$ containing the non-propagating
blocks so that the resulting tableaux are of size $n = 18$:
\begin{equation*}
    \scriptsize
    \squaresize=14pt
    \young{ \o4\o5 & \o7\o8 \cr \o2 \o3& \o6 \cr & & & & &&&&&&& & \o1& \o9 \cr}\,,
    ~
    \young{5& 9\cr 234 &16\cr  & & & & & &&&&&&& & 78 \cr }\,.
\end{equation*}
\end{example}

\begin{theorem} \label{thm:settableaux}
Let $n\geqslant 2k$. The set partitions of $[k]\cup[\o{k}]$ are in bijection
with pairs $(T,S)$ of standard multiset tableaux satisfying: $T$ and $S$ are of
the same shape $\lambda$, where $\lambda$ is a partition of $n$; $T$ has
content $[\o{k}]$; and $S$ has content $[k]$.
\end{theorem}

\begin{proof}
Let $\pi = \{\pi_1, \pi_2, \ldots, \pi_r\} \vdash [k] \cup [\o{k}]$
and let $(T, S)$ denote the tableaux constructed by the above procedure.
Notice that the propagation number of $\pi$ is at most $k$, and since $n\geqslant
2k$, there will be at least $k$ empty cells in the first row of $T$ and in the
first row of $S$ guaranteeing that both are semistandard multiset tableaux.
In addition, the cells of $T$ are filled with the blocks of a set partition of
$[\overline{k}]$ and the cells of $S$ with the blocks of a set partition of
$[k]$, and so there are no repetitions in $T$ or $S$. Hence both tableaux are standard multiset tableaux. 

Observe that we can reconstruct the set partition $\pi$ from $(T,S)$:
non-propagating blocks are the elements of the first rows;
and the inverse of the RSK algorithm recovers the generalized permutation
defined by the propagating blocks of $\pi$.
\end{proof}

Since the number of set partitions of a set of cardinality $2k$ is equal to the
Bell number $B(2k)$~\cite[A000110, A020557]{OEIS},
Theorem~\ref{thm:settableaux} implies that for $n \geqslant 2k$,
\begin{equation}\label{eq:bellsquares}
    B(2k) = \sum_{\lambda \vdash n} \#\SMT(\lambda,k)^2~.
\end{equation}

\begin{example}\label{ex:P2n}
    Figure~\ref{fig:P2(4)} depicts the correspondence of
    Theorem~\ref{thm:settableaux} for the 15 diagrams for $P_2(4)$.
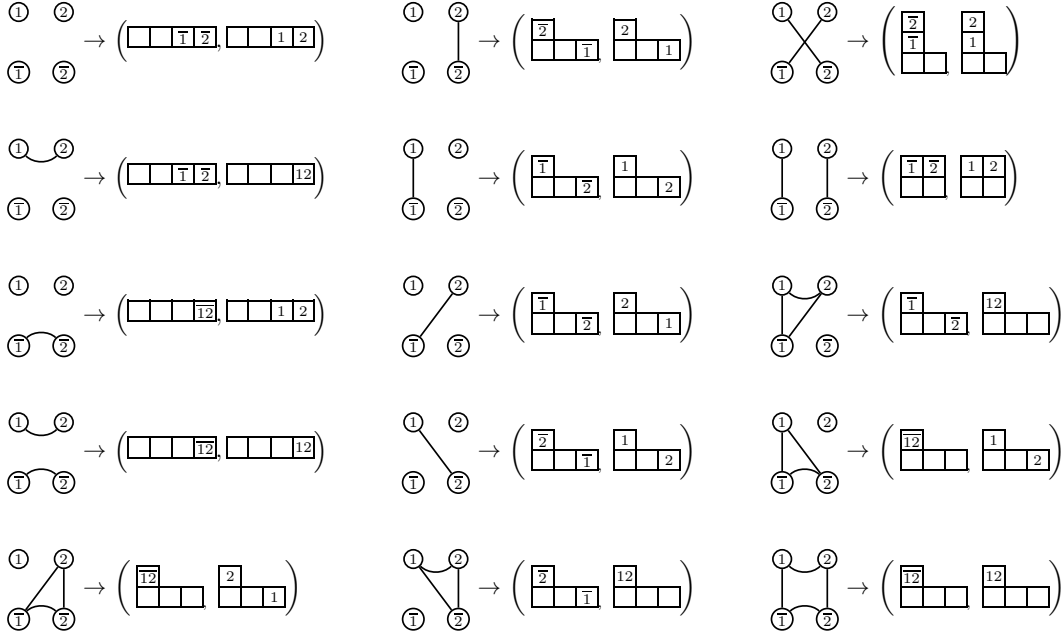
\begin{figure}[ht]
\scalebox{.8}{
    \begin{tabular}{l@{\hskip35pt}l@{\hskip35pt}l}
\begin{tikzpicture}[scale = 0.5,thick, baseline={(0,-1ex/2)}]
\tikzstyle{vertex} = [shape = circle, minimum size = 7pt, inner sep = 1pt]
\node[vertex] (G--2) at (1.5, -1) [shape = circle, draw] {\tiny $\overline 2$};
\node[vertex] (G--1) at (0.0, -1) [shape = circle, draw] {\tiny $\overline 1$};
\node[vertex] (G-1) at (0.0, 1) [shape = circle, draw] {\tiny $1$};
\node[vertex] (G-2) at (1.5, 1) [shape = circle, draw] {\tiny $2$};
\end{tikzpicture}
{$\rightarrow$}
$\left(\tiny{\young{ & & \o1& \o2 \cr}, \young{ & & 1 & 2 \cr}} \right)$
&
\begin{tikzpicture}[scale = 0.5,thick, baseline={(0,-1ex/2)}]
\tikzstyle{vertex} = [shape = circle, minimum size = 7pt, inner sep = 1pt]
\node[vertex] (G--2) at (1.5, -1) [shape = circle, draw] {\tiny $\overline 2$};
\node[vertex] (G-2) at (1.5, 1) [shape = circle, draw] {\tiny $2$};
\node[vertex] (G--1) at (0.0, -1) [shape = circle, draw] {\tiny $\overline 1$};
\node[vertex] (G-1) at (0.0, 1) [shape = circle, draw] {\tiny $1$};
\draw (G-2) .. controls +(0, -1) and +(0, 1) .. (G--2);
\end{tikzpicture}
{$\rightarrow$}
$\left(\raisebox{-6pt}{\tiny{\young{\o2 \cr & &  \o1 \cr  }, \young{ 2\cr  & &  1 \cr }}} \right)$
&
\begin{tikzpicture}[scale = 0.5,thick, baseline={(0,-1ex/2)}]
\tikzstyle{vertex} = [shape = circle, minimum size = 7pt, inner sep = 1pt]
\node[vertex] (G--2) at (1.5, -1) [shape = circle, draw] {\tiny $\overline 2$};
\node[vertex] (G-1) at (0.0, 1) [shape = circle, draw] {\tiny $1$};
\node[vertex] (G--1) at (0.0, -1) [shape = circle, draw] {\tiny $\overline 1$};
\node[vertex] (G-2) at (1.5, 1) [shape = circle, draw] {\tiny $2$};
\draw (G-1) .. controls +(0.75, -1) and +(-0.75, 1) .. (G--2);
\draw (G-2) .. controls +(-0.75, -1) and +(0.75, 1) .. (G--1);
\end{tikzpicture}
{$\rightarrow$}
$\left(\raisebox{-12pt}{\tiny{\young{ \o2 \cr \o1 \cr &  \cr  }, \young{ 2\cr 1\cr &   \cr }}} \right)$
\\\\\\
\begin{tikzpicture}[scale = 0.5,thick, baseline={(0,-1ex/2)}]
\tikzstyle{vertex} = [shape = circle, minimum size = 7pt, inner sep = 1pt]
\node[vertex] (G--2) at (1.5, -1) [shape = circle, draw] {\tiny $\overline 2$};
\node[vertex] (G--1) at (0.0, -1) [shape = circle, draw] {\tiny $\overline 1$};
\node[vertex] (G-1) at (0.0, 1) [shape = circle, draw] {\tiny $1$};
\node[vertex] (G-2) at (1.5, 1) [shape = circle, draw] {\tiny $2$};
\draw (G-1) .. controls +(0.5, -0.5) and +(-0.5, -0.5) .. (G-2);
\end{tikzpicture}
{$\rightarrow$}
$\left(\tiny{\young{ & & \o1& \o2 \cr}, \young{ & & & 12 \cr}} \right)$&
\begin{tikzpicture}[scale = 0.5,thick, baseline={(0,-1ex/2)}]
\tikzstyle{vertex} = [shape = circle, minimum size = 7pt, inner sep = 1pt]
\node[vertex] (G--2) at (1.5, -1) [shape = circle, draw] {\tiny $\overline 2$};
\node[vertex] (G--1) at (0.0, -1) [shape = circle, draw] {\tiny $\overline 1$};
\node[vertex] (G-1) at (0.0, 1) [shape = circle, draw] {\tiny $1$};
\node[vertex] (G-2) at (1.5, 1) [shape = circle, draw] {\tiny $2$};
\draw (G-1) .. controls +(0, -1) and +(0, 1) .. (G--1);
\end{tikzpicture}
{$\rightarrow$}
$\left(\raisebox{-6pt}{\tiny{\young{ \o1 \cr & &  \o2 \cr  }, \young{ 1\cr & &  2 \cr }}} \right)$&
\begin{tikzpicture}[scale = 0.5,thick, baseline={(0,-1ex/2)}]
\tikzstyle{vertex} = [shape = circle, minimum size = 7pt, inner sep = 1pt]
\node[vertex] (G--2) at (1.5, -1) [shape = circle, draw] {\tiny $\overline 2$};
\node[vertex] (G-2) at (1.5, 1) [shape = circle, draw] {\tiny $2$};
\node[vertex] (G--1) at (0.0, -1) [shape = circle, draw] {\tiny $\overline 1$};
\node[vertex] (G-1) at (0.0, 1) [shape = circle, draw] {\tiny $1$};
\draw (G-2) .. controls +(0, -1) and +(0, 1) .. (G--2);
\draw (G-1) .. controls +(0, -1) and +(0, 1) .. (G--1);
\end{tikzpicture}
{$\rightarrow$}
$\left(\raisebox{-6pt}{\tiny{\young{ \o1 &  \o2 \cr  &  \cr  }, \young{ 1& 2\cr &   \cr }}} \right)$\\\\\\
\begin{tikzpicture}[scale = 0.5,thick, baseline={(0,-1ex/2)}]
\tikzstyle{vertex} = [shape = circle, minimum size = 7pt, inner sep = 1pt]
\node[vertex] (G--2) at (1.5, -1) [shape = circle, draw] {\tiny $\overline 2$};
\node[vertex] (G--1) at (0.0, -1) [shape = circle, draw] {\tiny $\overline 1$};
\node[vertex] (G-1) at (0.0, 1) [shape = circle, draw] {\tiny $1$};
\node[vertex] (G-2) at (1.5, 1) [shape = circle, draw] {\tiny $2$};
\draw (G--2) .. controls +(-0.5, 0.5) and +(0.5, 0.5) .. (G--1);
\end{tikzpicture}
{$\rightarrow$}
$\left(\tiny{\young{ & & & \o1\o2 \cr }, \young{  & & 1& 2 \cr}} \right)$&
\begin{tikzpicture}[scale = 0.5,thick, baseline={(0,-1ex/2)}]
\tikzstyle{vertex} = [shape = circle, minimum size = 7pt, inner sep = 1pt]
\node[vertex] (G--2) at (1.5, -1) [shape = circle, draw] {\tiny $\overline 2$};
\node[vertex] (G--1) at (0.0, -1) [shape = circle, draw] {\tiny $\overline 1$};
\node[vertex] (G-2) at (1.5, 1) [shape = circle, draw] {\tiny $2$};
\node[vertex] (G-1) at (0.0, 1) [shape = circle, draw] {\tiny $1$};
\draw (G-2) .. controls +(-0.75, -1) and +(0.75, 1) .. (G--1);
\end{tikzpicture}
{$\rightarrow$}
$\left(\raisebox{-6pt}{\tiny{\young{\o1 \cr & &  \o2 \cr  }, \young{ 2 \cr & &  1 \cr }}} \right)$&
\begin{tikzpicture}[scale = 0.5,thick, baseline={(0,-1ex/2)}]
\tikzstyle{vertex} = [shape = circle, minimum size = 7pt, inner sep = 1pt]
\node[vertex] (G--2) at (1.5, -1) [shape = circle, draw] {\tiny $\overline 2$};
\node[vertex] (G--1) at (0.0, -1) [shape = circle, draw] {\tiny $\overline 1$};
\node[vertex] (G-1) at (0.0, 1) [shape = circle, draw] {\tiny $1$};
\node[vertex] (G-2) at (1.5, 1) [shape = circle, draw] {\tiny $2$};
\draw (G-1) .. controls +(0.5, -0.5) and +(-0.5, -0.5) .. (G-2);
\draw (G-2) .. controls +(-0.75, -1) and +(0.75, 1) .. (G--1);
\draw (G--1) .. controls +(0, 1) and +(0, -1) .. (G-1);
\end{tikzpicture}
{$\rightarrow$}
$\left(\raisebox{-6pt}{\tiny{\young{ \o1 \cr & & \o2 \cr  }, \young{12\cr  & &   \cr }}} \right)$
\\\\\\
\begin{tikzpicture}[scale = 0.5,thick, baseline={(0,-1ex/2)}]
\tikzstyle{vertex} = [shape = circle, minimum size = 7pt, inner sep = 1pt]
\node[vertex] (G--2) at (1.5, -1) [shape = circle, draw] {\tiny $\overline 2$};
\node[vertex] (G--1) at (0.0, -1) [shape = circle, draw] {\tiny $\overline 1$};
\node[vertex] (G-1) at (0.0, 1) [shape = circle, draw] {\tiny $1$};
\node[vertex] (G-2) at (1.5, 1) [shape = circle, draw] {\tiny $2$};
\draw (G--2) .. controls +(-0.5, 0.5) and +(0.5, 0.5) .. (G--1);
\draw (G-1) .. controls +(0.5, -0.5) and +(-0.5, -0.5) .. (G-2);
\end{tikzpicture}
{$\rightarrow$}
$\left(\tiny{\young{ & & & \o1\o2 \cr }, \young{  & & & 12 \cr}} \right)$&
\begin{tikzpicture}[scale = 0.5,thick, baseline={(0,-1ex/2)}]
\tikzstyle{vertex} = [shape = circle, minimum size = 7pt, inner sep = 1pt]
\node[vertex] (G--2) at (1.5, -1) [shape = circle, draw] {\tiny $\overline 2$};
\node[vertex] (G-1) at (0.0, 1) [shape = circle, draw] {\tiny $1$};
\node[vertex] (G--1) at (0.0, -1) [shape = circle, draw] {\tiny $\overline 1$};
\node[vertex] (G-2) at (1.5, 1) [shape = circle, draw] {\tiny $2$};
\draw (G-1) .. controls +(0.75, -1) and +(-0.75, 1) .. (G--2);
\end{tikzpicture}
{$\rightarrow$}
$\left(\raisebox{-6pt}{\tiny{\young{ \o2 \cr & &  \o1 \cr }, \young{ 1\cr  & &  2 \cr }}} \right)$&
\begin{tikzpicture}[scale = 0.5,thick, baseline={(0,-1ex/2)}]
\tikzstyle{vertex} = [shape = circle, minimum size = 7pt, inner sep = 1pt]
\node[vertex] (G--2) at (1.5, -1) [shape = circle, draw] {\tiny $\overline 2$};
\node[vertex] (G--1) at (0.0, -1) [shape = circle, draw] {\tiny $\overline 1$};
\node[vertex] (G-1) at (0.0, 1) [shape = circle, draw] {\tiny $1$};
\node[vertex] (G-2) at (1.5, 1) [shape = circle, draw] {\tiny $2$};
\draw (G-1) .. controls +(0.75, -1) and +(-0.75, 1) .. (G--2);
\draw (G--2) .. controls +(-0.5, 0.5) and +(0.5, 0.5) .. (G--1);
\draw (G--1) .. controls +(0, 1) and +(0, -1) .. (G-1);
\end{tikzpicture}
{$\rightarrow$}
$\left(\raisebox{-6pt}{\tiny{\young{ \o1\o2 \cr & &  \cr  }, \young{1\cr  & & 2  \cr }}} \right)$\\\\\\
\begin{tikzpicture}[scale = 0.5,thick, baseline={(0,-1ex/2)}]
\tikzstyle{vertex} = [shape = circle, minimum size = 7pt, inner sep = 1pt]
\node[vertex] (G--2) at (1.5, -1) [shape = circle, draw] {\tiny $\overline 2$};
\node[vertex] (G--1) at (0.0, -1) [shape = circle, draw] {\tiny $\overline 1$};
\node[vertex] (G-2) at (1.5, 1) [shape = circle, draw] {\tiny $2$};
\node[vertex] (G-1) at (0.0, 1) [shape = circle, draw] {\tiny $1$};
\draw (G-2) .. controls +(0, -1) and +(0, 1) .. (G--2);
\draw (G--2) .. controls +(-0.5, 0.5) and +(0.5, 0.5) .. (G--1);
\draw (G--1) .. controls +(0.75, 1) and +(-0.75, -1) .. (G-2);
\end{tikzpicture}
{$\rightarrow$}
$\left(\raisebox{-6pt}{\tiny{\young{ \o1\o2 \cr & &  \cr  }, \young{ 2\cr  & & 1  \cr }}} \right)$
&\begin{tikzpicture}[scale = 0.5,thick, baseline={(0,-1ex/2)}]
\tikzstyle{vertex} = [shape = circle, minimum size = 7pt, inner sep = 1pt]
\node[vertex] (G--2) at (1.5, -1) [shape = circle, draw] {\tiny $\overline 2$};
\node[vertex] (G-1) at (0.0, 1) [shape = circle, draw] {\tiny $1$};
\node[vertex] (G-2) at (1.5, 1) [shape = circle, draw] {\tiny $2$};
\node[vertex] (G--1) at (0.0, -1) [shape = circle, draw] {\tiny $\overline 1$};
\draw (G-1) .. controls +(0.5, -0.5) and +(-0.5, -0.5) .. (G-2);
\draw (G-2) .. controls +(0, -1) and +(0, 1) .. (G--2);
\draw (G--2) .. controls +(-0.75, 1) and +(0.75, -1) .. (G-1);
\end{tikzpicture}
{$\rightarrow$}
$\left(\raisebox{-6pt}{\tiny{\young{\o2 \cr  & & \o1 \cr  }, \young{12\cr  & &   \cr }}} \right)$&
\begin{tikzpicture}[scale = 0.5,thick, baseline={(0,-1ex/2)}]
\tikzstyle{vertex} = [shape = circle, minimum size = 7pt, inner sep = 1pt]
\node[vertex] (G--2) at (1.5, -1) [shape = circle, draw] {\tiny $\overline 2$};
\node[vertex] (G--1) at (0.0, -1) [shape = circle, draw] {\tiny $\overline 1$};
\node[vertex] (G-1) at (0.0, 1) [shape = circle, draw] {\tiny $1$};
\node[vertex] (G-2) at (1.5, 1) [shape = circle, draw] {\tiny $2$};
\draw (G-1) .. controls +(0.5, -0.5) and +(-0.5, -0.5) .. (G-2);
\draw (G-2) .. controls +(0, -1) and +(0, 1) .. (G--2);
\draw (G--2) .. controls +(-0.5, 0.5) and +(0.5, 0.5) .. (G--1);
\draw (G--1) .. controls +(0, 1) and +(0, -1) .. (G-1);
\end{tikzpicture}
{$\rightarrow$}
$\left(\raisebox{-6pt}{\tiny{\young{\o1\o2\cr & &  \cr  }, \young{ 12 \cr & &   \cr }}} \right)$
\end{tabular}
}
\caption{The correspondence from Theorem~\ref{thm:settableaux} for the 15
diagrams for $P_2(4)$.}
\label{fig:P2(4)}
\end{figure}
\end{example}

For any diagram $\pi \vdash [k]\cup[\o{k}]$, we define $\mbox{flip}(\pi)$ to be the reflection of $\pi$ along its horizontal axis. If
\begin{center}$\pi = $
\begin{tikzpicture}[scale = 0.5,thick, baseline={(0,-1ex/2)}]
\tikzstyle{vertex} = [shape = circle, minimum size = 7pt, inner sep = 1pt]
\node[vertex] (G--5) at (6.0, -1) [shape = circle, draw] {\tiny $\overline 5$};
\node[vertex] (G--4) at (4.5, -1) [shape = circle, draw] {\tiny $\overline 4$};
\node[vertex] (G--3) at (3.0, -1) [shape = circle, draw] {\tiny $\overline 3$};
\node[vertex] (G-5) at (6.0, 1) [shape = circle, draw] {\tiny 5};
\node[vertex] (G-4) at (4.5, 1) [shape = circle, draw] {\tiny 4};
\node[vertex] (G--2) at (1.5, -1) [shape = circle, draw] {\tiny $\overline 2$};
\node[vertex] (G--1) at (0.0, -1) [shape = circle, draw] {\tiny $\overline 1$};
\node[vertex] (G-1) at (0.0, 1) [shape = circle, draw] {\tiny 1};
\node[vertex] (G-2) at (1.5, 1) [shape = circle, draw] {\tiny 2};
\node[vertex] (G-3) at (3.0, 1) [shape = circle, draw] {\tiny 3};
\draw (G-5) .. controls +(-0.75, -1) and +(0.75, 1) .. (G--4);
\draw (G--4) .. controls +(-0.5, 0.5) and +(0.5, 0.5) .. (G--3);
\draw (G--3) .. controls +(1, 1) and +(-1, -1) .. (G-5);
\draw (G-4) .. controls +(-1, -1) and +(1, 1) .. (G--2);
\draw (G--2) .. controls +(-0.5, 0.5) and +(0.5, 0.5) .. (G--1);
\draw (G--1) .. controls +(1, 1) and +(-1, -1) .. (G-4);
\draw (G-1) .. controls +(0.5, -0.5) and +(-0.5, -0.5) .. (G-2);
\end{tikzpicture} \qquad then \qquad
$\mbox{flip}(\pi) = $
\begin{tikzpicture}[scale = 0.5,thick, baseline={(0,-1ex/2)}]
\tikzstyle{vertex} = [shape = circle, minimum size = 7pt, inner sep = 1pt]
\node[vertex] (G--5) at (6.0, -1) [shape = circle, draw] {\tiny $\overline 5$};
\node[vertex] (G--4) at (4.5, -1) [shape = circle, draw] {\tiny $\overline 4$};
\node[vertex] (G--3) at (3.0, -1) [shape = circle, draw] {\tiny $\overline 3$};
\node[vertex] (G-5) at (6.0, 1) [shape = circle, draw] {\tiny 5};
\node[vertex] (G-4) at (4.5, 1) [shape = circle, draw] {\tiny 4};
\node[vertex] (G--2) at (1.5, -1) [shape = circle, draw] {\tiny $\overline 2$};
\node[vertex] (G--1) at (0.0, -1) [shape = circle, draw] {\tiny $\overline 1$};
\node[vertex] (G-1) at (0.0, 1) [shape = circle, draw] {\tiny 1};
\node[vertex] (G-2) at (1.5, 1) [shape = circle, draw] {\tiny 2};
\node[vertex] (G-3) at (3.0, 1) [shape = circle, draw] {\tiny 3};
\draw (G--5) .. controls +(-0.5, .5) and +(0.5, -.5) .. (G-4);
\draw (G-3) .. controls +(0.5, -0.5) and +(-0.5,- 0.5) .. (G-4);
\draw (G-3) .. controls +(0.5, -0.5) and +(-.5, .5) .. (G--5);
\draw (G--4) .. controls +(-.5, .5) and +(.5, -.5) .. (G-2);
\draw (G-1) .. controls +(0.3, -0.3) and +(-0.3, -0.3) .. (G-2);
\draw (G-1) .. controls +(.5, -.5) and +(-.5, .5) .. (G--4);
\draw (G--2) .. controls +(-0.5, 0.5) and +(0.5, 0.5) .. (G--1);
\end{tikzpicture}
\end{center}
\vskip .2in

The properties in the next proposition follow directly from the RSK algorithm.

\begin{prop}\label{prop:properties-partitions}
Let $\pi \vdash [k]\cup[\o{k}]$.
\begin{enumerate}
\item[(a)] If $\pi$ inserts to $(T, S)$ with $S$ and $T$ of shape $\lambda$,
    then $|\overline{\lambda}| = \operatorname{pr}(\pi)$.
\item[(b)] If $\pi$ inserts to $(T,S)$, then $\operatorname{flip}(\pi)$ inserts to $(S, T)$.
\end{enumerate}
\end{prop}

\subsection{Restriction to subalgebras}
There are other bijections between partition algebra diagrams and pairs of
standard multiset tableaux, but an important aspect of the algorithm in this paper is that
it is compatible with the (representation theory) restriction to many prominent subalgebras of $P_k(n)$.
More precisely, we will see that this single procedure captures the
combinatorics of the representation theory of all these subalgebras.

For instance, for an integer $r$ with $0 \leqslant r \leqslant k$ the subspace
spanned by the set partitions with propagating number at most $r$ is
a subalgebra of $P_k(n)$
and the irreducible representations of this subalgebra are indexed by partitions
of size less than or equal to $r$.
By Proposition~\ref{prop:properties-partitions},
a refinement of Equation~\eqref{eq:bellsquares} states
\begin{equation}
    \# \Big\{ \pi \vdash [k] \cup [\o{k}] ~\Big|~ \operatorname{pr}(\pi) \leqslant r \Big\} =
    \sum_{\substack{\lambda \vdash n\\[1pt]|\o\lambda|\leqslant r}} \# \SMT(\lambda, k)^2~.
\end{equation}

\subsubsection{Definition of the subalgebras}

\begin{figure}[ht!]
\begin{tabular}{cccc}
 \begin{tikzpicture}[scale = 0.5,thick, baseline={(0,-1ex/2)}]
\tikzstyle{vertex} = [shape = circle, minimum size = 7pt, inner sep = 1pt]
\node[vertex] (G--4) at (4.5, -1) [shape = circle, draw] {\tiny $\overline 4$};
\node[vertex] (G--3) at (3.0, -1) [shape = circle, draw] {\tiny $\overline 3$};
\node[vertex] (G-4) at (4.5, 1) [shape = circle, draw] {\tiny 4};
\node[vertex] (G--2) at (1.5, -1) [shape = circle, draw] {\tiny $\overline 2$};
\node[vertex] (G--1) at (0.0, -1) [shape = circle, draw] {\tiny $\overline 1$};
\node[vertex] (G-1) at (0.0, 1) [shape = circle, draw] {\tiny 1};
\node[vertex] (G-2) at (1.5, 1) [shape = circle, draw] {\tiny 2};
\node[vertex] (G-3) at (3.0, 1) [shape = circle, draw] {\tiny 3};
\draw (G-1)--(G--2);
\draw (G-2)--(G--4);
\draw (G-3)--(G--1);
\draw (G-4)--(G--3);
\end{tikzpicture}{\hskip .2in}&{\hskip .3in}
\begin{tikzpicture}[scale = 0.5,thick, baseline={(0,-1ex/2)}]
\tikzstyle{vertex} = [shape = circle, minimum size = 7pt, inner sep = 1pt]
\node[vertex] (G--4) at (4.5, -1) [shape = circle, draw] {\tiny $\overline 4$};
\node[vertex] (G--3) at (3.0, -1) [shape = circle, draw] {\tiny $\overline 3$};
\node[vertex] (G-4) at (4.5, 1) [shape = circle, draw] {\tiny 4};
\node[vertex] (G--2) at (1.5, -1) [shape = circle, draw] {\tiny $\overline 2$};
\node[vertex] (G--1) at (0.0, -1) [shape = circle, draw] {\tiny $\overline 1$};
\node[vertex] (G-1) at (0.0, 1) [shape = circle, draw] {\tiny 1};
\node[vertex] (G-2) at (1.5, 1) [shape = circle, draw] {\tiny 2};
\node[vertex] (G-3) at (3.0, 1) [shape = circle, draw] {\tiny 3};
\draw (G-1) .. controls +(0.5, -0.7) and +(-0.5, -0.7) .. (G-3);
\draw (G-2)--(G--1);
\draw (G-4)--(G--3);
\draw (G--2) .. controls +(0.5, 0.7) and +(-0.5, 0.7) .. (G--4);
\end{tikzpicture}{\hskip .3in} &{\hskip .01in}
\begin{tikzpicture}[scale = 0.5,thick, baseline={(0,-1ex/2)}]
\tikzstyle{vertex} = [shape = circle, minimum size = 7pt, inner sep = 1pt]
\node[vertex] (G--4) at (4.5, -1) [shape = circle, draw] {\tiny $\overline 4$};
\node[vertex] (G--3) at (3.0, -1) [shape = circle, draw] {\tiny $\overline 3$};
\node[vertex] (G-4) at (4.5, 1) [shape = circle, draw] {\tiny 4};
\node[vertex] (G--2) at (1.5, -1) [shape = circle, draw] {\tiny $\overline 2$};
\node[vertex] (G--1) at (0.0, -1) [shape = circle, draw] {\tiny $\overline 1$};
\node[vertex] (G-1) at (0.0, 1) [shape = circle, draw] {\tiny 1};
\node[vertex] (G-2) at (1.5, 1) [shape = circle, draw] {\tiny 2};
\node[vertex] (G-3) at (3.0, 1) [shape = circle, draw] {\tiny 3};
\draw (G-1)--(G--3);
\draw (G-2)--(G--1);
\draw (G-4)--(G--4);
\end{tikzpicture}{\hskip .5in}&{\hskip .1in}
\begin{tikzpicture}[scale = 0.5,thick, baseline={(0,-1ex/2)}]
\tikzstyle{vertex} = [shape = circle, minimum size = 7pt, inner sep = 1pt]
\node[vertex] (G--4) at (4.5, -1) [shape = circle, draw] {\tiny $\overline 4$};
\node[vertex] (G--3) at (3.0, -1) [shape = circle, draw] {\tiny $\overline 3$};
\node[vertex] (G-4) at (4.5, 1) [shape = circle, draw] {\tiny 4};
\node[vertex] (G--2) at (1.5, -1) [shape = circle, draw] {\tiny $\overline 2$};
\node[vertex] (G--1) at (0.0, -1) [shape = circle, draw] {\tiny $\overline 1$};
\node[vertex] (G-1) at (0.0, 1) [shape = circle, draw] {\tiny 1};
\node[vertex] (G-2) at (1.5, 1) [shape = circle, draw] {\tiny 2};
\node[vertex] (G-3) at (3.0, 1) [shape = circle, draw] {\tiny 3};
\draw (G-1) .. controls +(0.5, -0.7) and +(-0.5, -0.7) .. (G-3);
\draw (G-2)--(G--1);
\draw (G-4)--(G--3);
\end{tikzpicture}\\[17pt]
permutation&perfect matching&partial permutation&matching
\\[15pt]
\begin{tikzpicture}[scale = 0.5,thick, baseline={(0,-1ex/2)}] 
\tikzstyle{vertex} = [shape = circle, minimum size = 7pt, inner sep = 1pt] 
\node[vertex] (G--4) at (4.5, -1) [shape = circle, draw] {\tiny $\overline 4$}; 
\node[vertex] (G--1) at (0.0, -1) [shape = circle, draw] {\tiny $\overline 1$}; 
\node[vertex] (G-4) at (4.5, 1) [shape = circle, draw] {\tiny $4$}; 
\node[vertex] (G--3) at (3.0, -1) [shape = circle, draw] {\tiny $\overline 3$}; 
\node[vertex] (G--2) at (1.5, -1) [shape = circle, draw] {\tiny $\overline 2$}; 
\node[vertex] (G-1) at (0.0, 1) [shape = circle, draw] {\tiny $1$}; 
\node[vertex] (G-3) at (3.0, 1) [shape = circle, draw] {\tiny $3$}; 
\node[vertex] (G-2) at (1.5, 1) [shape = circle, draw] {\tiny $2$}; 
\draw (G-4) .. controls +(0, -1) and +(0, 1) .. (G--4); 
\draw (G--4) .. controls +(-0.7, 0.7) and +(0.7, 0.7) .. (G--1); 
\draw (G--3) .. controls +(-0.5, 0.5) and +(0.5, 0.5) .. (G--2); 
\draw (G-1) .. controls +(0.6, -0.6) and +(-0.6, -0.6) .. (G-3); 
\end{tikzpicture}{\hskip .2in}&{\hskip .2in}
\begin{tikzpicture}[scale = 0.5,thick, baseline={(0,-1ex/2)}]
\tikzstyle{vertex} = [shape = circle, minimum size = 7pt, inner sep = 1pt]
\node[vertex] (G--4) at (4.5, -1) [shape = circle, draw] {\tiny $\overline 4$};
\node[vertex] (G--3) at (3.0, -1) [shape = circle, draw] {\tiny $\overline 3$};
\node[vertex] (G-4) at (4.5, 1) [shape = circle, draw] {\tiny 4};
\node[vertex] (G--2) at (1.5, -1) [shape = circle, draw] {\tiny $\overline 2$};
\node[vertex] (G--1) at (0.0, -1) [shape = circle, draw] {\tiny $\overline 1$};
\node[vertex] (G-1) at (0.0, 1) [shape = circle, draw] {\tiny 1};
\node[vertex] (G-2) at (1.5, 1) [shape = circle, draw] {\tiny 2};
\node[vertex] (G-3) at (3.0, 1) [shape = circle, draw] {\tiny 3};
\draw (G-1) .. controls +(0.5, -0.5) and +(-0.5, -0.5) .. (G-2);
\draw (G-3)--(G--1);
\draw (G-4)--(G--4);
\end{tikzpicture}{\hskip .3in} &{\hskip .01in}
\begin{tikzpicture}[scale = 0.5,thick, baseline={(0,-1ex/2)}]
\tikzstyle{vertex} = [shape = circle, minimum size = 7pt, inner sep = 1pt]
\node[vertex] (G--4) at (4.5, -1) [shape = circle, draw] {\tiny $\overline 4$};
\node[vertex] (G--3) at (3.0, -1) [shape = circle, draw] {\tiny $\overline 3$};
\node[vertex] (G-4) at (4.5, 1) [shape = circle, draw] {\tiny 4};
\node[vertex] (G--2) at (1.5, -1) [shape = circle, draw] {\tiny $\overline 2$};
\node[vertex] (G--1) at (0.0, -1) [shape = circle, draw] {\tiny $\overline 1$};
\node[vertex] (G-1) at (0.0, 1) [shape = circle, draw] {\tiny 1};
\node[vertex] (G-2) at (1.5, 1) [shape = circle, draw] {\tiny 2};
\node[vertex] (G-3) at (3.0, 1) [shape = circle, draw] {\tiny 3};
\draw (G-1) .. controls +(0.5, -0.5) and +(-0.5, -0.5) .. (G-2);
\draw (G-3)--(G--1);
\draw (G-4)--(G--2);
\draw (G--3) .. controls +(0.5, 0.5) and +(-0.5, 0.5) .. (G--4);
\end{tikzpicture}{\hskip .4in}&{\hskip .1in}
\begin{tikzpicture}[scale = 0.5,thick, baseline={(0,-1ex/2)}]
\tikzstyle{vertex} = [shape = circle, minimum size = 7pt, inner sep = 1pt]
\node[vertex] (G--4) at (4.5, -1) [shape = circle, draw] {\tiny $\overline 4$};
\node[vertex] (G--3) at (3.0, -1) [shape = circle, draw] {\tiny $\overline 3$};
\node[vertex] (G-4) at (4.5, 1) [shape = circle, draw] {\tiny 4};
\node[vertex] (G--2) at (1.5, -1) [shape = circle, draw] {\tiny $\overline 2$};
\node[vertex] (G--1) at (0.0, -1) [shape = circle, draw] {\tiny $\overline 1$};
\node[vertex] (G-1) at (0.0, 1) [shape = circle, draw] {\tiny 1};
\node[vertex] (G-2) at (1.5, 1) [shape = circle, draw] {\tiny 2};
\node[vertex] (G-3) at (3.0, 1) [shape = circle, draw] {\tiny 3};
\draw (G-2)--(G--1);
\draw (G-4)--(G--2);
\end{tikzpicture}
\\[17pt]
planar&planar matching&planar perfect&planar partial\\[-2pt]
&&matching&permutation
\end{tabular}
\caption{Examples of types of set partition diagrams.}
\label{fig:setpartitions}
\end{figure}
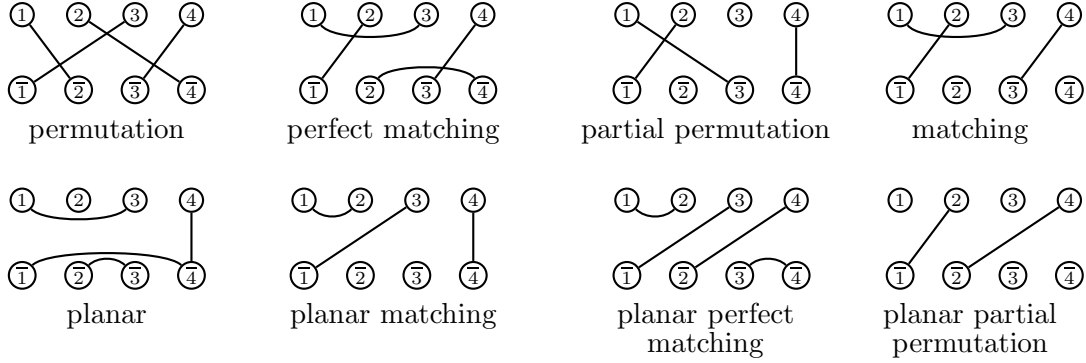

We introduce some terminology that will make it easier to define
the subalgebras.
See Figure~\ref{fig:setpartitions} for examples of the types of diagrams that
we define below.
A set partition $\pi$ is called \defn{planar}
if it can be represented as a graph without edge crossings inside the rectangle formed by its vertices.  A set partition is 
called a \defn{matching} if all its blocks are of size at most 2. We call a set partition a \defn{perfect matching} if all its 
blocks are of size 2.  The number of perfect matchings of $2n$ elements is equal
to $(2n-1)!! = (2n-1)(2n-3)\cdots(1)$.
A perfect matching, where each block contains an element in $[k]$ and an element in $[\overline{k}]$ 
is a \defn{permutation}.  A set partition is a \defn{partial permutation} if all its blocks have size one or two and 
every block of size two is propagating.

\begin{table}[b]
\captionof{table}{\small Subalgebras of the partition algebra $P_k(n)$.}
\label{tab:subalgebra}
\scalebox{.89}{
\centering
\renewcommand{\arraystretch}{1.4}
\begin{tabular} {p{70mm}|c|c}
\toprule
Subalgebra                                                                                               & Diagrams spanning the subalgebra & Dimension                                                                 \\ \midrule
Partition algebra $P_k(n)$                                                                               & all diagrams                     & $B(2k)$                                                                   \\
Group algebra of symmetric group $\mathbb{C} S_k$                                                        & permutations                     & $k!$                                                                      \\
Brauer algebra $B_k(n)$ \footnotesize \cite{Brauer, Wenzl}                                               & perfect matchings                & $(2k-1)!! $                                                               \\
Rook algebra $R_k(n)$ \footnotesize \cite{Solomon}                                                       & partial permutations             & $\sum\limits_{i=0}^k \smallbinom{k}{i}^2 i!$                              \\
Rook-Brauer algebra $R\!B_k(n)$ \footnotesize \cite{HD, MM}                                              & matchings                        & $\sum\limits_{i=0}^k \smallbinom{2k}{2i} (2i-1)!!$                        \\
Temperley--Lieb algebra $T\!L_k(n)$ \newline \null \hfill \footnotesize \cite{TL, Jones1, West, Martin1} & planar perfect matchings         & $\frac{1}{k+1} \smallbinom{2k}{k}$                                        \\
Motzkin algebra $M_k(n)$ \footnotesize \cite{BH-motzkin}                                                 & planar matchings                 & $\sum\limits_{i=0}^k\frac{1}{i+1} \smallbinom{2i}{i} \smallbinom{2k}{2i}$ \\
Planar rook algebra $P\!R_k(n)$ \footnotesize \cite{FHH}                                                 & planar partial permutations      & $\smallbinom{2k}{k}$                                                      \\
Planar algebra $P\!P_k(n)$ \footnotesize \cite{Jones}                                                    & planar diagrams                  & $\frac{1}{2k+1}\smallbinom{4k}{2k}$                                       \\
\bottomrule
\end{tabular}}
\end{table}

Table~\ref{tab:subalgebra} summarizes the definitions of
the subalgebras that we work with.
In \cite{HJ}, the authors construct the irreducible representations
of these subalgebras using standard multiset tableaux (which they call \emph{set-partition tableaux}) and compute their
characters. Their results provide a detailed study of the
representation theory of these subalgebras from which we extract the
information in Table~\ref{tab:irreducibles}.

\begin{table}[b]
\caption{\small Dimensions and index sets for irreducible representations of
    certain subalgebras $A_k$ of the partition algebra $P_k(n)$~\cite{HJ}.
    We highlight that the bottom four subalgebras are all planar and
    that their irreducible representations are indexed by partitions
    having a single part.}
\label{tab:irreducibles}
\scalebox{.90}{
\centering
\renewcommand{\arraystretch}{1.5}
\begin{tabular}{c|c|c}
\toprule
$A_k$ & Index set for irreducibles
      & Dimension of irreducible $V^\lambda_{A_k}$\\
\midrule
$P_k(n)$&$\{ \lambda \mid \lambda \vdash m, 0 \leqslant m \leqslant k\}$
        &$f^{{\lambda}} \sum\limits_{i=|\lambda|}^{k} \smallbinom{k}{i} \smallstirling{i}{|\lambda|} B(k-i)$\\
$\mathbb{C}S_k$ &$\{ \lambda \mid \lambda \vdash k\}$
                & $f^{\lambda}$ \\
$B_k(n)$& $\{\lambda\mid \lambda \vdash k-2r, 0 \leqslant 2r \leqslant k \}$ 
        & $f^{{\lambda}} \smallbinom{k}{|{\lambda}|} (k-|{\lambda}|-1)!!$\\
$R_k(n)$&$\{ \lambda \mid \lambda \vdash m, 0 \leqslant m \leqslant k\}$
        &
$f^{{\lambda}} \smallbinom{k}{|{\lambda}|}$\\
$R\!B_k(n)$& $\{ \lambda \mid \lambda \vdash m, 0 \leqslant m \leqslant k\}$
           &$f^{{\lambda}} \smallbinom{k}{|{\lambda}|}  \sum\limits_{i=0}^{(k-|{\lambda}|)/2} \smallbinom{k-|{\lambda}|}{2i} (2i-1)!!$\\
\midrule
$T\!L_k(n)$ &$\{ (k-2r) \mid 0 \leqslant 2r \leqslant k\}$
            & $\smallbinom{k}{ (k-m)/2} -\smallbinom{k}{ (k-m)/2-1}$\\
$M_k(n)$&$\{(m) \mid 0 \leqslant m \leqslant k\}$
        & $\sum\limits_{i=0}^{\lfloor(k-m)/2\rfloor} \smallbinom{{k}}{{m+2i}} \Big( \smallbinom{{m+2i}}{ {i}} -\smallbinom{{m+2i}}{{i-1}} \Big)$\\
$P\!R_k(n)$ &$\{(m) \mid 0 \leqslant m \leqslant k\}$
            & $\smallbinom{k}{ m}$\\
$P\!P_k(n)$&$\{(m) \mid 0 \leqslant m \leqslant k\}$
           & $\smallbinom{2k }{ k-m} -\smallbinom{2k}{ k-m-1}$\\
\bottomrule
\end{tabular}}
\end{table}

\subsubsection{Restricting the correspondence to the subalgebras}
We characterize the standard multiset tableaux produced by the
correspondence of Section~\ref{ssec:RSK-for-diagrams} when restricted to
the diagrams spanning one of the subalgebras $A_k$ in
Table~\ref{tab:subalgebra}. We denote this set by \defn{$\SMT_{A_k}(\lambda)$}.

A standard multiset tableau is \defn{matching} if the first row contains sets
of size less than or equal to $2$ and all other rows contain only sets of size
$1$. In Lemma~\ref{lemma:tableauxdesc}, we show these are the multiset tableaux
that correspond to matching diagrams by our insertion algorithm.

Two sets $S$ and $S'$ are \defn{non-crossing}
if there do not exist elements $a,b \in S$ and $c,d \in S'$ such that
either $a<c<b<d$ or $c<a<d<b$. We say that $c \in [k]$
is \defn{between} a set $S$ if there exist $a,b \in S$ such that $a<c<b$.
We call a standard multiset tableau \defn{planar} if it has two rows,
if the sets in the first row are pairwise non-crossing, and if no element belonging
to one of the sets in the second row is between any set in the tableau
(apart from the set containing the element).
In Lemma~\ref{lemma:tableauxdesc}, we show these are the multiset tableaux
that correspond to planar diagrams by our insertion algorithm.

\begin{lemma}
\label{lemma:tableauxdesc}
Let $k$ be any positive integer, $\lambda$ a partition of an integer $n$ with
$n \geqslant 2k$, and $A_k$ one of the subalgebras of $P_k(n)$ defined in
Table~\ref{tab:subalgebra}.
If we apply the insertion procedure of Theorem~\ref{thm:settableaux} to the
diagrams spanning $A_k$, then the resulting standard multiset
tableaux are characterized by the properties listed in Table~\ref{tab:characterisation_of_Ak-SMT}.
\end{lemma}

\begin{table}[b]
\caption{\small Properties characterizing the standard multiset tableaux
    that belong to $\SMT_{A_k}(\lambda)$. These are the tableaux produced by
    the correspondence of Section~\ref{ssec:RSK-for-diagrams} when restricted
    to the diagrams spanning $A_k$.}
\label{tab:characterisation_of_Ak-SMT}
\scalebox{.90}{
\centering
\renewcommand{\arraystretch}{1.2}
\begin{tabular}{lp{50mm}p{30mm}p{35mm}}
    \toprule
    & & \multicolumn{2}{c}{properties characterizing $\SMT_{A_k}$} \\ \cmidrule(lr){3-4}
    $A_k$ & diagrams spanning $A_k$ & sizes of entries \newline in first row & other properties \\
    \midrule
    $P_k(n)$        & all diagrams                & ---           & ---      \\
    $P\!P_k(n)$     & planar diagrams             & ---           & planar \\
    $\mathbb{C}S_k$ & permutations                & $0$           & matching \\
    $B_k(n)$        & perfect matchings           & $0$, $2$      & matching \\
    $R_k(n)$        & partial permutations        & $0$, $1$      & matching \\
    $R\!B_k(n)$     & matchings                   & $0$, $1$, $2$ & matching \\
    $T\!L_k(n)$     & planar perfect matchings    & $0$, $2$      & matching \emph{\&} planar \\
    $M_k(n)$        & planar matchings            & $0$, $1$, $2$ & matching \emph{\&} planar \\
    $P\!R_k(n)$     & planar partial permutations & $0$, $1$      & matching \emph{\&} planar \\
    \bottomrule
\end{tabular}
}
\end{table}

\begin{proof}
    The case $A_k = P_k(n)$ follows from Theorem~\ref{thm:settableaux}.

For $A_k = P\!P_k(n)$, observe that if a set partition $\pi$ is planar, then the propagating blocks $\pi_{i_1}, \ldots, \pi_{i_p}$ 
are inserted in order as the blocks are non-crossing. This means that the shape of the insertion
tableau $T$ (as well as that of the recording tableau $S$) has at most two rows. Also notice that non-propagating blocks have to be non-crossing
(since the diagram is planar) and these entries constitute the entries in the first row of $T$ and $S$.
Furthermore, propagating blocks in $\pi$ correspond to entries in the second row of $T$ and $S$.
If a letter in the second row of $T$ (respectively, $S$) is between another set
in $T$ (respectively, $S$), then the diagram $\pi$ is not planar.

By the correspondence described in Section~\ref{ssec:RSK-for-diagrams}, a propagating block of a diagram
$\pi$ is a set of size $2$ with one element in $[k]$ and one element in $[\o{k}]$
if and only if the corresponding entries in the pair of standard multiset
tableaux $(T,S)$ have size one and appear in the
second row or above in both $T$ and $S$.
The non-propagating blocks are all in the first row and, in a matching diagram, all
of the blocks are of size less than or equal to $2$.
This implies that, if $A_k$ is spanned by diagrams that are matching, then
these diagrams insert to tableaux which are matching.
Similarly, if the subalgebra is spanned by diagrams that are planar, then these
diagrams insert to tableaux which are planar.

Since the non-empty sets that appear in the first row of $T$ and the first row
of $S$ correspond to the non-propagating blocks of the set partition, we obtain
the restrictions on sizes of the sets appearing in the first row.
For instance, if $A_k = \CC S_k$, then there are no non-propagating blocks,
and so the first row of $S$ and of $T$ contain only empty sets.
If $A_k$ is $R_k(n)$ or $P\!R_k(n)$, then the non-propagating blocks are all of
size $1$.
If $A_k$ is $B_k(n)$ or $T\!L_k(n)$, then the blocks (and hence the
non-propagating blocks) are all of size $2$.
If $A_k$ is $R\!B_k(n)$ or $M_k(n)$, then non-propagating blocks are
of size at most $2$.
\end{proof}

\begin{example}
In Table~\ref{table.tableaux}, we give examples of the tableaux described in Lemma~\ref{lemma:tableauxdesc} for
$k=9$ and $n$ sufficiently large.
\end{example}

\begin{table}[b]
\caption{Examples of the tableaux in Lemma~\ref{lemma:tableauxdesc} for
$k=9$ and $n \geqslant 2k$.
\label{table.tableaux}}
\begin{center}
\scalebox{0.85}{
\begin{tabular}{ccll}
    \toprule
    Algebra & Diagram & Insertion Tableaux ($T$) & Recording Tableaux ($S$) \\ \midrule
$P_k(n)$ &
\begin{tikzpicture}[scale = 0.4,thick, baseline={(0,-1ex/2)}]
\tikzstyle{vertex} = [shape = circle, minimum size = 7pt, inner sep = 1pt]
\node[vertex] (G--9) at (12.0, -1) [shape = circle, draw] {\tiny $\overline 9$};
\node[vertex] (G--8) at (10.5, -1) [shape = circle, draw] {\tiny $\overline 8$};
\node[vertex] (G--7) at (9.0, -1) [shape = circle, draw] {\tiny $\overline 7$};
\node[vertex] (G--6) at (7.5, -1) [shape = circle, draw] {\tiny $\overline 6$};
\node[vertex] (G--5) at (6.0, -1) [shape = circle, draw] {\tiny $\overline 5$};
\node[vertex] (G--4) at (4.5, -1) [shape = circle, draw] {\tiny $\overline 4$};
\node[vertex] (G--3) at (3.0, -1) [shape = circle, draw] {\tiny $\overline 3$};
\node[vertex] (G--2) at (1.5, -1) [shape = circle, draw] {\tiny $\overline 2$};
\node[vertex] (G--1) at (0.0, -1) [shape = circle, draw] {\tiny $\overline 1$};
\node[vertex] (G-9) at (12.0, 1) [shape = circle, draw] {\tiny $9$};
\node[vertex] (G-8) at (10.5, 1) [shape = circle, draw] {\tiny $8$};
\node[vertex] (G-7) at (9.0, 1) [shape = circle, draw] {\tiny $7$};
\node[vertex] (G-6) at (7.5, 1) [shape = circle, draw] {\tiny $6$};
\node[vertex] (G-5) at (6.0, 1) [shape = circle, draw] {\tiny $5$};
\node[vertex] (G-4) at (4.5, 1) [shape = circle, draw] {\tiny $4$};
\node[vertex] (G-3) at (3.0, 1) [shape = circle, draw] {\tiny $3$};
\node[vertex] (G-2) at (1.5, 1) [shape = circle, draw] {\tiny $2$};
\node[vertex] (G-1) at (0.0, 1) [shape = circle, draw] {\tiny $1$};
\draw (G-1).. controls +(0.5, -0.5) and +(-0.5, -0.5) ..(G-2);
\draw (G-2).. controls +(0.5, -0.7) and +(-0.5, -0.7) ..(G-4);
\draw (G--2).. controls +(0.5, 0.8) and +(-0.5, 0.8) .. (G--6);
\draw (G-4)--(G--6);
\draw (G-1)--(G--2);
\draw (G-5).. controls +(0.5, -0.5) and +(-0.5, -0.5) ..(G-6);
\draw (G-6).. controls +(0.5, -0.5) and +(-0.5, -0.5) ..(G-7);
\draw (G--3).. controls +(0.5, 0.5) and +(-0.5, 0.5) .. (G--4);
\draw (G--4).. controls +(0.5, 0.5) and +(-0.5, 0.5) .. (G--5);
\draw (G-5).. controls +(0.5, -0.5) and +(-0.5, 0.5) ..(G--3);
\draw (G-7).. controls +(0.5, -0.5) and +(-0.5, 0.5) ..(G--5);
\draw (G--7).. controls +(0.5, 0.5) and +(-0.5, 0.5) .. (G--8);
\draw (G-8).. controls +(0.5, -0.5) and +(-0.5, -0.5) ..(G-9);
\draw (G-9)--(G--9);
\draw (G-8)--(G--9);
\end{tikzpicture} &
\begin{tikzpicture}[scale = 0.4,thick, baseline={(0,-1ex/2)}]
\node at (26,0) {$\tiny{\squaresize=12pt \young{\o2\o6 \cr \o3\o4\o5 & \o9 \cr & & & \compactcdots & & \o1 & \o7\o8 \cr}}$};
\end{tikzpicture} &
\begin{tikzpicture}[scale = 0.5,thick, baseline={(0,-1ex/2)}]
\node at (26,0) {$\tiny{\squaresize=12pt \young{567 \cr 124 & 89 \cr & & & \compactcdots & & & 3\cr}}$};
\end{tikzpicture}
\\ \midrule
$\mathbb{C}S_k$ &
\begin{tikzpicture}[scale = 0.4,thick, baseline={(0,-1ex/2)}]
\tikzstyle{vertex} = [shape = circle, minimum size = 7pt, inner sep = 1pt]
\node[vertex] (G--9) at (12.0, -1) [shape = circle, draw] {\tiny $\overline 9$};
\node[vertex] (G--8) at (10.5, -1) [shape = circle, draw] {\tiny $\overline 8$};
\node[vertex] (G--7) at (9.0, -1) [shape = circle, draw] {\tiny $\overline 7$};
\node[vertex] (G--6) at (7.5, -1) [shape = circle, draw] {\tiny $\overline 6$};
\node[vertex] (G--5) at (6.0, -1) [shape = circle, draw] {\tiny $\overline 5$};
\node[vertex] (G--4) at (4.5, -1) [shape = circle, draw] {\tiny $\overline 4$};
\node[vertex] (G--3) at (3.0, -1) [shape = circle, draw] {\tiny $\overline 3$};
\node[vertex] (G--2) at (1.5, -1) [shape = circle, draw] {\tiny $\overline 2$};
\node[vertex] (G--1) at (0.0, -1) [shape = circle, draw] {\tiny $\overline 1$};
\node[vertex] (G-9) at (12.0, 1) [shape = circle, draw] {\tiny $9$};
\node[vertex] (G-8) at (10.5, 1) [shape = circle, draw] {\tiny $8$};
\node[vertex] (G-7) at (9.0, 1) [shape = circle, draw] {\tiny $7$};
\node[vertex] (G-6) at (7.5, 1) [shape = circle, draw] {\tiny $6$};
\node[vertex] (G-5) at (6.0, 1) [shape = circle, draw] {\tiny $5$};
\node[vertex] (G-4) at (4.5, 1) [shape = circle, draw] {\tiny $4$};
\node[vertex] (G-3) at (3.0, 1) [shape = circle, draw] {\tiny $3$};
\node[vertex] (G-2) at (1.5, 1) [shape = circle, draw] {\tiny $2$};
\node[vertex] (G-1) at (0.0, 1) [shape = circle, draw] {\tiny $1$};
\draw (G-1)--(G--4);
\draw (G-2).. controls +(0.5, -0.7) and +(-0.5, 0.7) ..(G--6);
\draw (G-3)--(G--1);
\draw (G-4)--(G--2);
\draw (G-5)--(G--3);
\draw (G-6)--(G--7);
\draw (G-7)--(G--9);
\draw (G-8)--(G--8);
\draw (G-9).. controls +(0.5, -0.7) and +(-0.5, 0.7) ..(G--5);
\end{tikzpicture} &
\begin{tikzpicture}[scale = 0.4,thick, baseline={(0,-1ex/2)}]
\node at (24.5,0) {$\tiny{\young{\o9 \cr \o4 & \o6 & \o7 \cr \o1 & \o2 & \o3 & \o5 & \o8 \cr & & & & & & \compactcdots & \cr}}$}; 
\end{tikzpicture} &
\begin{tikzpicture}[scale = 0.4,thick, baseline={(0,-1ex/2)}]
\node at (24.5,0) {$\tiny{\young{9 \cr 3 & 4 & 8 \cr 1 & 2 & 5 & 6 & 7 \cr & & & & & & \compactcdots & \cr}}$};
\end{tikzpicture} \\ \midrule
$B_k(n)$ &
\begin{tikzpicture}[scale = 0.4,thick, baseline={(0,-1ex/2)}]
\tikzstyle{vertex} = [shape = circle, minimum size = 7pt, inner sep = 1pt]
\node[vertex] (G--9) at (12.0, -1) [shape = circle, draw] {\tiny $\overline 9$};
\node[vertex] (G--8) at (10.5, -1) [shape = circle, draw] {\tiny $\overline 8$};
\node[vertex] (G--7) at (9.0, -1) [shape = circle, draw] {\tiny $\overline 7$};
\node[vertex] (G--6) at (7.5, -1) [shape = circle, draw] {\tiny $\overline 6$};
\node[vertex] (G--5) at (6.0, -1) [shape = circle, draw] {\tiny $\overline 5$};
\node[vertex] (G--4) at (4.5, -1) [shape = circle, draw] {\tiny $\overline 4$};
\node[vertex] (G--3) at (3.0, -1) [shape = circle, draw] {\tiny $\overline 3$};
\node[vertex] (G--2) at (1.5, -1) [shape = circle, draw] {\tiny $\overline 2$};
\node[vertex] (G--1) at (0.0, -1) [shape = circle, draw] {\tiny $\overline 1$};
\node[vertex] (G-9) at (12.0, 1) [shape = circle, draw] {\tiny $9$};
\node[vertex] (G-8) at (10.5, 1) [shape = circle, draw] {\tiny $8$};
\node[vertex] (G-7) at (9.0, 1) [shape = circle, draw] {\tiny $7$};
\node[vertex] (G-6) at (7.5, 1) [shape = circle, draw] {\tiny $6$};
\node[vertex] (G-5) at (6.0, 1) [shape = circle, draw] {\tiny $5$};
\node[vertex] (G-4) at (4.5, 1) [shape = circle, draw] {\tiny $4$};
\node[vertex] (G-3) at (3.0, 1) [shape = circle, draw] {\tiny $3$};
\node[vertex] (G-2) at (1.5, 1) [shape = circle, draw] {\tiny $2$};
\node[vertex] (G-1) at (0.0, 1) [shape = circle, draw] {\tiny $1$};
\draw (G-1).. controls +(0.5, -0.7) and +(-0.5, -0.7) ..(G-4);
\draw (G-6).. controls +(0.5, -0.5) and +(-0.5, -0.5) ..(G-7);
\draw (G--2).. controls +(0.5, 0.5) and +(-0.5, 0.5) .. (G--3);
\draw (G--6).. controls +(0.5, 0.7) and +(-0.5, 0.7) .. (G--9);
\draw (G-2).. controls +(0.5, -0.7) and +(-0.5, 0.7) .. (G--8);
\draw (G-3)--(G--4);
\draw (G-5).. controls +(0.5, -0.7) and +(-0.5, 0.7) ..(G--1);
\draw (G-8)--(G--7);
\draw (G-9).. controls +(0.5, -0.7) and +(-0.5, 0.7) ..(G--5);
\end{tikzpicture} &
 \begin{tikzpicture}[scale = 0.4,thick, baseline={(0,-1ex/2)}]
 \node at (23,0) {$\tiny{\young{\o8 \cr \o4 & \o7 \cr \o1 & \o5 \cr & & & \compactcdots & & \o2\o3 & \o6\o9 \cr}}$};
\end{tikzpicture} &
 \begin{tikzpicture}[scale = 0.4,thick, baseline={(0,-1ex/2)}]
 \node at (23,0) {$\tiny{\young{5 \cr 3 & 9 \cr 2 & 8 \cr & & & \compactcdots & & 14 & 67 \cr}}$};
\end{tikzpicture} \\ \midrule
$R_k(n)$ &
\begin{tikzpicture}[scale = 0.4,thick, baseline={(0,-1ex/2)}]
\tikzstyle{vertex} = [shape = circle, minimum size = 7pt, inner sep = 1pt]
\node[vertex] (G--9) at (12.0, -1) [shape = circle, draw] {\tiny $\overline 9$};
\node[vertex] (G--8) at (10.5, -1) [shape = circle, draw] {\tiny $\overline 8$};
\node[vertex] (G--7) at (9.0, -1) [shape = circle, draw] {\tiny $\overline 7$};
\node[vertex] (G--6) at (7.5, -1) [shape = circle, draw] {\tiny $\overline 6$};
\node[vertex] (G--5) at (6.0, -1) [shape = circle, draw] {\tiny $\overline 5$};
\node[vertex] (G--4) at (4.5, -1) [shape = circle, draw] {\tiny $\overline 4$};
\node[vertex] (G--3) at (3.0, -1) [shape = circle, draw] {\tiny $\overline 3$};
\node[vertex] (G--2) at (1.5, -1) [shape = circle, draw] {\tiny $\overline 2$};
\node[vertex] (G--1) at (0.0, -1) [shape = circle, draw] {\tiny $\overline 1$};
\node[vertex] (G-9) at (12.0, 1) [shape = circle, draw] {\tiny $9$};
\node[vertex] (G-8) at (10.5, 1) [shape = circle, draw] {\tiny $8$};
\node[vertex] (G-7) at (9.0, 1) [shape = circle, draw] {\tiny $7$};
\node[vertex] (G-6) at (7.5, 1) [shape = circle, draw] {\tiny $6$};
\node[vertex] (G-5) at (6.0, 1) [shape = circle, draw] {\tiny $5$};
\node[vertex] (G-4) at (4.5, 1) [shape = circle, draw] {\tiny $4$};
\node[vertex] (G-3) at (3.0, 1) [shape = circle, draw] {\tiny $3$};
\node[vertex] (G-2) at (1.5, 1) [shape = circle, draw] {\tiny $2$};
\node[vertex] (G-1) at (0.0, 1) [shape = circle, draw] {\tiny $1$};
\draw (G-3).. controls +(0.5, -0.7) and +(-0.5, 0.7) ..(G--8);
\draw (G-5)--(G--4);
\draw (G-6).. controls +(0.5, -0.7) and +(-0.5, 0.7) ..(G--2);
\draw (G-7).. controls +(0.5, -0.7) and +(-0.5, 0.7) ..(G--3);
\draw (G-9)--(G--7);
\end{tikzpicture} &
 \begin{tikzpicture}[scale = 0.4,thick, baseline={(0,-1ex/2)}]
\node at (26.5,0) {$\tiny{\young{\o8 \cr \o4 \cr \o2 & \o3  & \o7 \cr & & & & \compactcdots & & \o1 & \o5 & \o6 & \o9 \cr}}$};
\end{tikzpicture} &
 \begin{tikzpicture}[scale = 0.4,thick, baseline={(0,-1ex/2)}]
\node at (26.5,0) {$\tiny{\young{6 \cr 5 \cr 3 & 7 & 9 \cr & & & & \compactcdots & & 1 & 2 & 4 & 8 \cr}}$};
\end{tikzpicture} \\ \midrule
\raisebox{-.3in}{\rule{0pt}{10ex}}$R\!B_k(n)$ &
\begin{tikzpicture}[scale = 0.4,thick, baseline={(0,-1ex/2)}]
\tikzstyle{vertex} = [shape = circle, minimum size = 7pt, inner sep = 1pt]
\node[vertex] (G--9) at (12.0, -1) [shape = circle, draw] {\tiny $\overline 9$};
\node[vertex] (G--8) at (10.5, -1) [shape = circle, draw] {\tiny $\overline 8$};
\node[vertex] (G--7) at (9.0, -1) [shape = circle, draw] {\tiny $\overline 7$};
\node[vertex] (G--6) at (7.5, -1) [shape = circle, draw] {\tiny $\overline 6$};
\node[vertex] (G--5) at (6.0, -1) [shape = circle, draw] {\tiny $\overline 5$};
\node[vertex] (G--4) at (4.5, -1) [shape = circle, draw] {\tiny $\overline 4$};
\node[vertex] (G--3) at (3.0, -1) [shape = circle, draw] {\tiny $\overline 3$};
\node[vertex] (G--2) at (1.5, -1) [shape = circle, draw] {\tiny $\overline 2$};
\node[vertex] (G--1) at (0.0, -1) [shape = circle, draw] {\tiny $\overline 1$};
\node[vertex] (G-9) at (12.0, 1) [shape = circle, draw] {\tiny $9$};
\node[vertex] (G-8) at (10.5, 1) [shape = circle, draw] {\tiny $8$};
\node[vertex] (G-7) at (9.0, 1) [shape = circle, draw] {\tiny $7$};
\node[vertex] (G-6) at (7.5, 1) [shape = circle, draw] {\tiny $6$};
\node[vertex] (G-5) at (6.0, 1) [shape = circle, draw] {\tiny $5$};
\node[vertex] (G-4) at (4.5, 1) [shape = circle, draw] {\tiny $4$};
\node[vertex] (G-3) at (3.0, 1) [shape = circle, draw] {\tiny $3$};
\node[vertex] (G-2) at (1.5, 1) [shape = circle, draw] {\tiny $2$};
\node[vertex] (G-1) at (0.0, 1) [shape = circle, draw] {\tiny $1$};
\draw (G-2).. controls +(0.5, -0.5) and +(-0.5, -0.5) ..(G-3);
\draw (G-5).. controls +(0.5, -0.8) and +(-0.5, -0.8) ..(G-8);
\draw (G--2).. controls +(0.5, 0.8) and +(-0.5, 0.8) .. (G--7);
\draw (G--4).. controls +(0.5, 0.6) and +(-0.5, 0.6) .. (G--6);
\draw (G-1)--(G--3);
\draw (G-4)--(G--1);
\draw (G-7)--(G--5);
\end{tikzpicture} &
\begin{tikzpicture}[scale = 0.4,thick, baseline={(0,-1ex/2)}]
\node at (25,0) {$\tiny{\young{\o3 \cr \o1 & \o5 \cr & & & \compactcdots & &  \o4\o6 & \o2\o7 & \o8 & \o9 \cr}}$};
\end{tikzpicture} &
\begin{tikzpicture}[scale = 0.4,thick, baseline={(0,-1ex/2)}]
\node at (25,0) {$\tiny{\young{4 \cr  1 & 7 \cr & & & \compactcdots & & 23 & 6 & 58 & 9 \cr}}$};
\end{tikzpicture} \\ \midrule
\raisebox{-.3in}{\rule{0pt}{10ex}}$T\!L_k(n)$ &
\begin{tikzpicture}[scale = 0.4,thick, baseline={(0,-1ex/2)}]
\tikzstyle{vertex} = [shape = circle, minimum size = 7pt, inner sep = 1pt]
\node[vertex] (G--9) at (12.0, -1) [shape = circle, draw] {\tiny $\overline 9$};
\node[vertex] (G--8) at (10.5, -1) [shape = circle, draw] {\tiny $\overline 8$};
\node[vertex] (G--7) at (9.0, -1) [shape = circle, draw] {\tiny $\overline 7$};
\node[vertex] (G--6) at (7.5, -1) [shape = circle, draw] {\tiny $\overline 6$};
\node[vertex] (G--5) at (6.0, -1) [shape = circle, draw] {\tiny $\overline 5$};
\node[vertex] (G--4) at (4.5, -1) [shape = circle, draw] {\tiny $\overline 4$};
\node[vertex] (G--3) at (3.0, -1) [shape = circle, draw] {\tiny $\overline 3$};
\node[vertex] (G--2) at (1.5, -1) [shape = circle, draw] {\tiny $\overline 2$};
\node[vertex] (G--1) at (0.0, -1) [shape = circle, draw] {\tiny $\overline 1$};
\node[vertex] (G-9) at (12.0, 1) [shape = circle, draw] {\tiny $9$};
\node[vertex] (G-8) at (10.5, 1) [shape = circle, draw] {\tiny $8$};
\node[vertex] (G-7) at (9.0, 1) [shape = circle, draw] {\tiny $7$};
\node[vertex] (G-6) at (7.5, 1) [shape = circle, draw] {\tiny $6$};
\node[vertex] (G-5) at (6.0, 1) [shape = circle, draw] {\tiny $5$};
\node[vertex] (G-4) at (4.5, 1) [shape = circle, draw] {\tiny $4$};
\node[vertex] (G-3) at (3.0, 1) [shape = circle, draw] {\tiny $3$};
\node[vertex] (G-2) at (1.5, 1) [shape = circle, draw] {\tiny $2$};
\node[vertex] (G-1) at (0.0, 1) [shape = circle, draw] {\tiny $1$};
\draw (G-1).. controls +(0.5, -0.5) and +(-0.5, -0.5) ..(G-2);
\draw (G-3).. controls +(0.5, -0.5) and +(-0.5, -0.5) ..(G-4);
\draw (G-7).. controls +(0.5, -0.5) and +(-0.5, -0.5) .. (G-8);
\draw (G--4).. controls +(0.5, 0.5) and +(-0.5, 0.5) .. (G--5);
\draw (G--3).. controls +(0.5, 0.8) and +(-0.5, 0.8) .. (G--6);
\draw (G--8).. controls +(0.5, 0.5) and +(-0.5, 0.5) .. (G--9);
\draw (G-5).. controls +(0.5, -0.7) and +(-0.5, 0.7) ..(G--1);
\draw (G-6).. controls +(0.5, -0.7) and +(-0.5, 0.7) ..(G--2);
\draw (G-9)--(G--7);
\end{tikzpicture} &
\begin{tikzpicture}[scale = 0.4,thick, baseline={(0,-1ex/2)}]
\node at (25,0) {$\tiny{\young{\o1 & \o2 & \o7 \cr & & & & \compactcdots & &  \o4\o5 & \o3\o6 & \o8\o9 \cr}}$};
\end{tikzpicture} &
\begin{tikzpicture}[scale = 0.4,thick, baseline={(0,-1ex/2)}]
\node at (25,0) {$\tiny{\young{5 & 6 & 9 \cr & & & & \compactcdots & & 12 & 34 & 78 \cr}}$};
\end{tikzpicture} \\ \midrule
\raisebox{-.3in}{\rule{0pt}{10ex}}$M_k(n)$ &
\begin{tikzpicture}[scale = 0.4,thick, baseline={(0,-1ex/2)}]
\tikzstyle{vertex} = [shape = circle, minimum size = 7pt, inner sep = 1pt]
\node[vertex] (G--9) at (12.0, -1) [shape = circle, draw] {\tiny $\overline 9$};
\node[vertex] (G--8) at (10.5, -1) [shape = circle, draw] {\tiny $\overline 8$};
\node[vertex] (G--7) at (9.0, -1) [shape = circle, draw] {\tiny $\overline 7$};
\node[vertex] (G--6) at (7.5, -1) [shape = circle, draw] {\tiny $\overline 6$};
\node[vertex] (G--5) at (6.0, -1) [shape = circle, draw] {\tiny $\overline 5$};
\node[vertex] (G--4) at (4.5, -1) [shape = circle, draw] {\tiny $\overline 4$};
\node[vertex] (G--3) at (3.0, -1) [shape = circle, draw] {\tiny $\overline 3$};
\node[vertex] (G--2) at (1.5, -1) [shape = circle, draw] {\tiny $\overline 2$};
\node[vertex] (G--1) at (0.0, -1) [shape = circle, draw] {\tiny $\overline 1$};
\node[vertex] (G-9) at (12.0, 1) [shape = circle, draw] {\tiny $9$};
\node[vertex] (G-8) at (10.5, 1) [shape = circle, draw] {\tiny $8$};
\node[vertex] (G-7) at (9.0, 1) [shape = circle, draw] {\tiny $7$};
\node[vertex] (G-6) at (7.5, 1) [shape = circle, draw] {\tiny $6$};
\node[vertex] (G-5) at (6.0, 1) [shape = circle, draw] {\tiny $5$};
\node[vertex] (G-4) at (4.5, 1) [shape = circle, draw] {\tiny $4$};
\node[vertex] (G-3) at (3.0, 1) [shape = circle, draw] {\tiny $3$};
\node[vertex] (G-2) at (1.5, 1) [shape = circle, draw] {\tiny $2$};
\node[vertex] (G-1) at (0.0, 1) [shape = circle, draw] {\tiny $1$};
\draw (G-1).. controls +(0.5, -0.5) and +(-0.5, -0.5) ..(G-2);
\draw (G-3).. controls +(0.5, -0.5) and +(-0.5, -0.5) ..(G-4);
\draw (G--1).. controls +(0.5, 0.7) and +(-0.5, 0.7) .. (G--3);
\draw (G-5)--(G--4);
\draw (G-7)--(G--5);
\draw (G-8)--(G--6);
\draw (G-9)--(G--8);
\end{tikzpicture} &
\begin{tikzpicture}[scale = 0.4,thick, baseline={(0,-1ex/2)}]
\node at (26,0) {$\tiny{\young{\o4 & \o5 & \o6 & \o8 \cr & & & & &  \compactcdots & & \o2 & \o1\o3 & \o7 & \o9 \cr}}$};
\end{tikzpicture} &
\begin{tikzpicture}[scale = 0.4,thick, baseline={(0,-1ex/2)}]
\node at (26,0) {$\tiny{\young{5 & 7 & 8 & 9 \cr & & & & & \compactcdots & & & 12 & 34 & 6 \cr}}$};
\end{tikzpicture} \\ \midrule
\raisebox{-.3in}{\rule{0pt}{10ex}}$P\!R_k(n)$&
\begin{tikzpicture}[scale = 0.4,thick, baseline={(0,-1ex/2)}]
\tikzstyle{vertex} = [shape = circle, minimum size = 7pt, inner sep = 1pt]
\node[vertex] (G--9) at (12.0, -1) [shape = circle, draw] {\tiny $\overline 9$};
\node[vertex] (G--8) at (10.5, -1) [shape = circle, draw] {\tiny $\overline 8$};
\node[vertex] (G--7) at (9.0, -1) [shape = circle, draw] {\tiny $\overline 7$};
\node[vertex] (G--6) at (7.5, -1) [shape = circle, draw] {\tiny $\overline 6$};
\node[vertex] (G--5) at (6.0, -1) [shape = circle, draw] {\tiny $\overline 5$};
\node[vertex] (G--4) at (4.5, -1) [shape = circle, draw] {\tiny $\overline 4$};
\node[vertex] (G--3) at (3.0, -1) [shape = circle, draw] {\tiny $\overline 3$};
\node[vertex] (G--2) at (1.5, -1) [shape = circle, draw] {\tiny $\overline 2$};
\node[vertex] (G--1) at (0.0, -1) [shape = circle, draw] {\tiny $\overline 1$};
\node[vertex] (G-9) at (12.0, 1) [shape = circle, draw] {\tiny $9$};
\node[vertex] (G-8) at (10.5, 1) [shape = circle, draw] {\tiny $8$};
\node[vertex] (G-7) at (9.0, 1) [shape = circle, draw] {\tiny $7$};
\node[vertex] (G-6) at (7.5, 1) [shape = circle, draw] {\tiny $6$};
\node[vertex] (G-5) at (6.0, 1) [shape = circle, draw] {\tiny $5$};
\node[vertex] (G-4) at (4.5, 1) [shape = circle, draw] {\tiny $4$};
\node[vertex] (G-3) at (3.0, 1) [shape = circle, draw] {\tiny $3$};
\node[vertex] (G-2) at (1.5, 1) [shape = circle, draw] {\tiny $2$};
\node[vertex] (G-1) at (0.0, 1) [shape = circle, draw] {\tiny $1$};
\draw (G-2)--(G--1);
\draw (G-4)--(G--2);
\draw (G-5)--(G--4);
\draw (G-7)--(G--5);
\draw (G-8)--(G--7);
\draw (G-9)--(G--8);
\end{tikzpicture} &
\begin{tikzpicture}[scale = 0.4,thick, baseline={(0,-1ex/2)}]
\node at (26,0) {$\tiny{\young{\o1 & \o2 & \o4 & \o5  & \o7  & \o8 \cr & & & & & & & \compactcdots & & \o3 & \o6 & \o9 \cr}}$};
\end{tikzpicture} &
\begin{tikzpicture}[scale = 0.4,thick, baseline={(0,-1ex/2)}]
\node at (26,0) {$\tiny{\young{2 & 4 & 5 & 7 & 8 & 9 \cr & & & & & & &  \compactcdots & & 1 & 3 & 6 \cr}}$};
\end{tikzpicture} \\ \midrule
\raisebox{-.3in}{\rule{0pt}{10ex}}$P\!P_k(n)$ &
\begin{tikzpicture}[scale = 0.4,thick, baseline={(0,-1ex/2)}]
\tikzstyle{vertex} = [shape = circle, minimum size = 7pt, inner sep = 1pt]
\node[vertex] (G--9) at (12.0, -1) [shape = circle, draw] {\tiny $\overline 9$};
\node[vertex] (G--8) at (10.5, -1) [shape = circle, draw] {\tiny $\overline 8$};
\node[vertex] (G--7) at (9.0, -1) [shape = circle, draw] {\tiny $\overline 7$};
\node[vertex] (G--6) at (7.5, -1) [shape = circle, draw] {\tiny $\overline 6$};
\node[vertex] (G--5) at (6.0, -1) [shape = circle, draw] {\tiny $\overline 5$};
\node[vertex] (G--4) at (4.5, -1) [shape = circle, draw] {\tiny $\overline 4$};
\node[vertex] (G--3) at (3.0, -1) [shape = circle, draw] {\tiny $\overline 3$};
\node[vertex] (G--2) at (1.5, -1) [shape = circle, draw] {\tiny $\overline 2$};
\node[vertex] (G--1) at (0.0, -1) [shape = circle, draw] {\tiny $\overline 1$};
\node[vertex] (G-9) at (12.0, 1) [shape = circle, draw] {\tiny $9$};
\node[vertex] (G-8) at (10.5, 1) [shape = circle, draw] {\tiny $8$};
\node[vertex] (G-7) at (9.0, 1) [shape = circle, draw] {\tiny $7$};
\node[vertex] (G-6) at (7.5, 1) [shape = circle, draw] {\tiny $6$};
\node[vertex] (G-5) at (6.0, 1) [shape = circle, draw] {\tiny $5$};
\node[vertex] (G-4) at (4.5, 1) [shape = circle, draw] {\tiny $4$};
\node[vertex] (G-3) at (3.0, 1) [shape = circle, draw] {\tiny $3$};
\node[vertex] (G-2) at (1.5, 1) [shape = circle, draw] {\tiny $2$};
\node[vertex] (G-1) at (0.0, 1) [shape = circle, draw] {\tiny $1$};
\draw (G-3).. controls +(0.5, -0.5) and +(-0.5, -0.5) ..(G-4);
\draw (G-1).. controls +(0.5, -0.7) and +(-0.5, -0.7) ..(G-3);
\draw (G-4).. controls +(0.5, -0.5) and +(-0.5, 0.5) .. (G--1);
\draw (G-1)--(G--1);
\draw (G--2).. controls +(0.5, 0.5) and +(-0.5, 0.5) .. (G--3);
\draw (G--3).. controls +(0.5, 0.5) and +(-0.5, 0.5) .. (G--4);
\draw (G--5).. controls +(0.5, 0.5) and +(-0.5, 0.5) .. (G--6);
\draw (G--6).. controls +(0.5, 0.5) and +(-0.5, 0.5) .. (G--7);
\draw (G-5)--(G--5);
\draw (G-5).. controls +(0.5, -0.5) and +(-0.5, 0.5) .. (G--7);
\draw (G-8).. controls +(0.5, -0.5) and +(-0.5, -0.5) ..(G-9);
\draw (G-7)--(G--9);
\end{tikzpicture} &
\begin{tikzpicture}[scale = 0.4,thick, baseline={(0,-1ex/2)}]
\node at (26,0) {$\tiny{\squaresize=12pt\young{\o1 & \o5\o6\o7 & \o9 \cr & & & & \compactcdots & &  & \o2\o3\o4 & \o8 \cr}}$};
\end{tikzpicture} &
\begin{tikzpicture}[scale = 0.5,thick, baseline={(0,-1ex/2)}]
\node at (26,0) {$\tiny{\squaresize=12pt\young{134 & 5 & 7 \cr & & & & \compactcdots & & 2 & 6 & 89\cr}}$};
\end{tikzpicture} \\ \bottomrule
\end{tabular}}
\end{center}
\end{table}

In addition, we obtain the following corollary of Theorem~\ref{thm:settableaux}
and Lemma~\ref{lemma:tableauxdesc}.

\begin{cor}
\label{alg-dim}
If $n \geqslant 2k$, then for each subalgebra $A_k$ of the partition algebra
$P_k(n)$ described in Table~\ref{tab:subalgebra}, we have
 \[
     \dim(A_k) = \sum_{\lambda \vdash n} \big(\# \SMT_{A_k}(\lambda)\big)^2,
 \]
where the dimension of $A_k$ is also given in Table~\ref{tab:subalgebra}.
\end{cor}

\subsection{From standard multiset tableaux to Bratteli diagrams}
\label{section.bratteli}
Let $A_k$ denote one of the subalgebras from Lemma~\ref{lemma:tableauxdesc}.
We establish a bijection between the standard multiset tableaux for
$A_k$ and the paths in the Bratteli diagram for $A_0 \subseteq A_1 \subseteq A_2 \subseteq \cdots$.

A \defn{Bratteli diagram} associated to a tower of algebras
$A_0 \subseteq A_1 \subseteq A_2 \subseteq \cdots$
is an infinite $\NN$-graded graph defined as follows.
The vertices at level $k \in \NN$ are in bijection with the isomorphism classes
of the irreducible representations of $A_k$; if the irreducible representations
are parameterized by some index set, then we label the vertices by the elements
of the index set. Note that it is possible that vertices at different levels
carry the same label (this happens for some of the index sets listed in
Table~\ref{tab:irreducibles}), but the associated representations are
different.
The edges in the graph connect vertices of level $k$ with vertices at level $k+1$:
the number of edges from
the vertex associated with an irreducible $A_{k}$-representation $V$
to
the vertex associated with an irreducible $A_{k+1}$-representation $V'$
is the multiplicity of
$V$ in the restriction of $V'$ to $A_{k}$.

In all the examples we consider, there is exactly one irreducible
$A_0$-representation and it is of dimension $1$. It follows from an induction
argument that the dimension of an irreducible representation $V$ is equal to
the number of paths in the Bratteli diagram from the unique level-$0$ vertex to
the vertex associated with $V$.

\begin{example}Young's lattice is an example of a Bratteli diagram
for the tower of symmetric group algebras ${\mathbb C}S_0 \subseteq {\mathbb C}S_1
\subseteq {\mathbb C}S_2 \subseteq \cdots$.
Indeed, recall that there is exactly one edge in Young's lattice
\defn{$\mu \rightarrow \lambda$} if and only if $\mu$ is obtained from $\lambda$ by
removing a corner cell. And, the multiplicity of the irreducible
$S_{k}$-representation indexed by $\mu$ in the restriction to $\CC S_{k}$ of
the irreducible $S_{k+1}$-representation indexed by $\lambda$ is equal to $1$
if $\mu \rightarrow \lambda$ in Young's lattice and is equal to $0$ otherwise.
\end{example}

By \defn{branching rule}, we mean any combinatorial description of the edge
multiplicities in the Bratteli diagram in terms of the index sets of the
irreducible representations.
Table~\ref{tab:branching} summarizes the branching
rule for various subalgebras of the partition algebra, where the index
sets for the irreducible representations are given in Table~\ref{tab:irreducibles}.

\begin{table}[b]
\begin{minipage}{\linewidth}
\centering
\captionof{table}{Branching rules for various subalgebras of the partition algebra} \label{tab:branching}
\begin{tabular}{p{15mm}p{0.5\textwidth}p{45mm}}
\toprule
$A_k$    & Branching rule for $\lambda$ of $A_k$ to $\mu$ of $A_{k-1}$       & Reference \\ \midrule
$P_k(n)$ \newline $P\!P_k(n)$ &
{remove $1$ or $0$ cells from $\lambda$ to get $\tau$, \newline then add $1$ or $0$ cells to $\tau$ to get $\mu$} &
\cite[Equation (1.4.1)]{Hal} \newline \cite[Section 1]{HalRam} \\
\cmidrule(rl){1-3}
$\mathbb{C}S_k $ & remove a cell from $\lambda$ to get $\mu$ & \\
\cmidrule(rl){1-3}
$B_k(n)$ & \multirow{2}{*}{add or remove a cell from $\lambda$ to get $\mu$}                  & \cite[p. 192]{Wenzl} \\
$T\!L_k(n)$ & & \cite[p. 19]{Jones1} \\
\cmidrule(rl){1-3}
$R_k(n)$    & \multirow{2}{*}{remove $1$ or $0$ cells from $\lambda$ to get $\mu$} & \cite[Sec. 3.1]{Halverson} \\
$P\!R_k(n)$ & & \cite[Equation (4)]{FHH} \\
\cmidrule(rl){1-3}
$R\!B_k(n)$ & \multirow{2}{*}{add or remove $1$ or $0$ cells from $\lambda$ to get $\mu$} & \cite[Equation (3.4)]{HD} \\
$M_k(n)$ & & \cite[Equation (3.12)]{BH-motzkin} \\
\bottomrule
\end{tabular}
\end{minipage}
\end{table}

\begin{remark}
A proof that the planar algebra $P\!P_k(n)$ is isomorphic to the Temperley--Lieb algebra $T\!L_{2k}(n)$
can be found in \cite[Section 1]{HalRam}. Consequently, the branching rule for $P\!P_k(n)$
is obtained by a repeated application of the branching rule for $T\!L_{2k}(n)$.

Paths in the Brauer algebra Bratteli diagram are often called updown tableaux or oscillating
tableaux in the literature; see \cite{HalLew} and the references therein.
\end{remark}

\begin{prop}\label{prop:onestep}
Let $k$ be a positive integer and $\lambda$ a partition of $n \geqslant 2k$.
There is a bijection
\[
    \phi \colon \SMT(\lambda, k) \to
    \Big\{ \big(S, \tau\big) ~\Big|~
    S \in \SMT(\mu, k-1), \,
    \tau \rightarrow \mu, \,
    \tau \rightarrow \lambda\Big\}~.
\]
\end{prop}

\begin{proof}
Let $T$ be an element of $\SMT(\lambda,k)$.
Let $A$ denote the unique set appearing in $T$ containing $k$.
Because we are using last letter order, the cell labelled by $A$ is a corner
cell.
Let $T'$ be the tableau obtained from $T$ by deleting $A$.
If $A \neq \{k\}$, then let $S$ be the tableau obtained by inserting $A
\setminus \{k\}$ in the second row of $T'$ using the RSK insertion procedure.
If $A = \{k\}$, then let $S$ be the tableau obtained from $T'$ by adding
a blank cell at the beginning of its first row.
Set $\phi(T) = (S,\tau)$, where $\tau$ is the shape of $T'$.
Note that $\tau \rightarrow \lambda$.

Conversely, let $(S, \tau)$ be such that $S \in \SMT(\mu, k-1)$,
$\tau \rightarrow \mu$ and $\tau \rightarrow \lambda$.
If the unique cell in $\mu / \tau$ is in the first row, then let
$T$ be the tableau obtained from $S$ by removing a blank cell from the first
row of $S$, and then adding a cell labelled $\{k\}$ at the cell in $\lambda/\tau$.

Otherwise, let $S'$ denote the tableau obtained from $S$ by deleting its first
row. Reverse the RSK insertion procedure starting with the cell $\mu/\tau$ to
produce a tableau $T'$ and a set $A'$ such that inserting $A'$ into $T'$
produces $S'$.
Let $T$ be the tableau obtained from $T'$ by adjoining the first row of $S$ and
adding a new cell labelled $A' \cup \{k\}$ at the cell in $\lambda/\tau$.
\end{proof}

This correspondence is particularly useful because it respects the properties
characterizing the tableaux in $\SMT_{A_k}(\lambda)$ (see Table~\ref{tab:characterisation_of_Ak-SMT}).

\begin{example}\label{ex:brauertab} Consider the following tableau in $\SMT((n-3,2,1), 9)$,
\squaresize=12pt
\[
    T=\raisebox{-.2in}{\tiny{ \young{7\cr 5 & 8 \cr & & & \compactcdots& & 12 & 46 & 39 \cr}}}~,
\]
in particular, it is an element of $\SMT_{B_9(n)}((n-3,2,1))$.
To compute $\phi(T) = (S, \tau)$, we remove the cell labelled $\{3,9\}$
and insert $\{3\}$ in the second row, obtaining
\[
    S=\raisebox{-.25in}{\tiny{ \young{7\cr 5\cr 3 & 8 \cr & & & \compactcdots& & 12 & 46\cr}}}
\]
where $S \in \SMT_{B_8(n)}((n-4,2,1,1))$ and $\tau = (n-4,2,1)$.
Proposition~\ref{prop:onestep} also states that $T$ can be recovered from
$S$ and the partitions $\tau = (n-4,2,1)$ and $\lambda = (n-3, 2, 1)$.
\end{example}

Now we are ready for the main result of this section, which states that the standard multiset tableaux
in $\SMT_{A_k}(\lambda)$ encode the branching rule for the subalgebra $A_k$.

\begin{theorem}\label{th:edges}
Let $A_k$ be any of the subalgebras of $P_k(n)$ defined in Table~\ref{tab:subalgebra}
with $n \geqslant 2k$, and $\lambda,\mu \vdash n$.
\begin{enumerate}
    \item
        If $T \in \SMT_{A_{k}}(\lambda)$ and $\phi(T)= (S, \tau)$, then
        $S \in \SMT_{A_{k-1}}(\mu)$.

    \item
        For each $S \in \SMT_{A_{k-1}}(\mu)$,
        the number of $T \in \SMT_{A_{k}}(\lambda)$
        such that $\phi(T) = (S,\tau)$ for some partition $\tau$
        is equal to the number of edges from $\o\mu$ to $\o\lambda$ in the Bratteli diagram
        for the tower of algebras $A_0 \subseteq A_1 \subseteq A_2 \subseteq \cdots$.
\end{enumerate}
\end{theorem}

\begin{proof}
(1)
We first verify that if $T$ is planar (respectively, matching) and
$\phi(T)=(S,\tau)$, then $S$ is also planar (respectively, matching).

\smallskip

Let $T$ be a planar tableau and let $A$ denote the set in $T$ that contains $k$.
If $A = \{k\}$, then $S$ is obtained from $T$ by deleting the cell labelled $A$
and adding a blank cell to the first row.
Since $T$ is planar, all the sets appearing in $T$ satisfy the conditions in
the definition of planar, and so $S$ is also planar.

Suppose $A \neq \{k\}$ and that $A$ appears in the second row of $T$.
Let $a$ be the largest element in $A \setminus \{k\}$.
Then $a+1, a+2, \ldots, k-1$ must be in the first row of $T$ (otherwise,
these elements are between $A$, contradicting that $T$ is planar).
Therefore, $A \setminus \{k\}$ is greater than all the sets in the second row
of $T$ in the last letter order. Thus, $S$ is obtained from $T$ by deleting
$k$, and it follows that $S$ is planar.

Suppose $A \neq \{k\}$ and that $A$ appears in the first row of $T$.
If one of the sets in the second row of $T$ contains $c \in [k]$ satisfying
$\max(A \setminus \{k\}) < c < k$, then $c$ is between $A$, which
contradicts the hypothesis that $T$ is planar. Hence, $A \setminus \{k\}$ is
greater than all sets appearing in the second row of $T$,
and so $S$ is obtained from $T$ by deleting the cell labelled $A$
and appending $A \setminus \{k\}$ to the second row.
To prove that $S$ is planar, it remains to show that no element of $A \setminus
\{k\}$ is between any other set in $S$.

Suppose there exists $b \in A \setminus \{k\}$ that is between some set $B$.
Then there exist $a, c \in B$ such that $a < b < c < k$.
If $B$ is in the first row of $S$, then $A$ and $B$ are crossing, which
contradicts the fact that the sets in the first row of $T$ are pairwise
non-crossing.
If $B$ is in the second row of $S$, then $c$ is between $A$, which contradicts
the fact that no element belonging to the second row of $T$ is between any
set in the tableau.
Hence, $S$ is also planar.

\smallskip

Let $T$ be a matching tableau and let $A$ denote the set in $T$ that contains
$k$.
If $A = \{k\}$, then $S$ is obtained from $T$ by deleting the cell labelled $A$
and adding a blank cell to the first row.
Since $T$ is matching, all the sets appearing in $T$ satisfy the conditions in
the definition of matching, and so $S$ is also matching.

Suppose $A \neq \{k\}$. Then $A$ is a set of size $2$, say $A = \{a, k\}$,
and it appears in the first row of $T$. Then $S$ is obtained from $T$
by deleting $A$ and inserting $\{a\}$ into the second row using the RSK
algorithm. Thus, all sets in $S$ are of size at most $2$, and the sets of size
$2$ belong to its first row. Hence, $S$ is matching.

\smallskip

Finally, note that if the sizes of the sets in the first row of $T$ are
constrained to be in some set $\varSigma$ that contains $0$, then the same
is true for the sets in the first for of $S$: indeed, $S$ is obtained from $T$
by deleting a set and either adding a blank cell in the first row or by adding
a non-empty cell in some row besides the first row.

\medskip

(2)
Next we check that for a fixed $S \in \SMT_{A_{k-1}}(\mu)$,
$$\# \{ T \in \SMT_{A_{k}}(\lambda) \mid \phi(T) = (S,\tau) \hbox{ for some }\tau\}$$
is the multiplicity of $V^{\o\mu}_{A_{k-1}}$ in the restriction
of $V^{\o\lambda}_{A_k}$ to $A_{k-1}$ described in Table~\ref{tab:branching}.
We will do this on a case by case basis for each of the four pairs of subalgebras.
Throughout this proof, let $S \in \SMT_{A_{k-1}}(\mu)$ and let $T \in
\SMT_{A_{k}}(\lambda)$ be such that $\phi(T) = (S, \tau)$ for some $\tau$.

\smallskip
\emph{Let $A_k$ be either $R_k(n)$ or $P\!R_k(n)$.}
Note that all the sets appearing in $S$ are of size at most $1$,
and $S$ is obtained from $T$ by deleting
the cell labelled $k$ and adding an empty cell to the first row.
The cell labelled $k$ is removed from the first row if and only if $\o\lambda = \o\mu$.
And if the cell is removed from some other row, then $\o\mu$ is obtained from
$\o\lambda$ by deleting a cell.
This agrees with the branching rule in Table~\ref{tab:branching}
with $\lambda$ replaced by $\o{\lambda}$ and $\mu$ replaced by $\o{\mu}$.

\smallskip
\emph{Let $A_k$ be $B_k(n)$ or $T\!L_k(n)$.}
The non-empty
sets in the first row of $S$ are all of size $2$ and the sets in the other rows
are all of size $1$.
If the set of $T$ containing $k$ is $\{k\}$, then it appears
in the second row or above of $T$.
In this case, $S$ is obtained from $T$ by deleting the cell labelled $\{k\}$
and adding an empty cell in the first row.
Thus, $\mu$ is obtained from $\lambda$ by moving a cell to the first row,
or in other words, $\o{\mu}$ is obtained from $\o{\lambda}$ by deleting a cell.

Otherwise, the set containing $k$ is $\{a,k\}$, for some $a$, and it appears at
the end of the first row. Then $S$ is obtained from $T$ by deleting
$\{a,k\}$ and inserting $\{a\}$ in the second row.
Thus, $\mu$ is obtained from $\lambda$ by removing a cell from the first row
and adding a cell to some other row.
In other words, $\o{\mu}$ is obtained from $\o{\lambda}$ by adding a cell.
This agrees with the branching rule in Table~\ref{tab:branching}
with $\lambda$ replaced by $\o{\lambda}$ and $\mu$ replaced by $\o{\mu}$.

\smallskip
\emph{Let $A_k$ be $R\!B_k(n)$ or $M_k(n)$.}
The non-empty sets in the first row of $S$ are all of size $1$ or $2$ and those
in the other rows are all of size $1$.
Hence, the set in $T$ containing $k$ is either $\{k\}$ or $\{a, k\}$ for some $a$.
In the first case, $S$ is obtained from $T$ by deleting the cell labelled $\{k\}$
and adding an empty cell to the first row, from which it follows that we have
$\o\lambda = \o\mu$ (when $\{k\}$ is in the first row of $T$) or
$\o\mu \rightarrow \o\lambda$ (otherwise).
In the second case, $S$ is obtained from $T$ by deleting the cell labelled
$\{a,k\}$ at the end of the first row and using the RSK insertion procedure to
insert $\{a\}$ into the second row. Thus, $\mu$ is obtained from $\lambda$ by
moving a cell from the first row to some other row. In other words, $\o\mu$ is
obtained from $\o\lambda$ by adding a cell.
This agrees with the branching rule in Table~\ref{tab:branching}
with $\lambda$ replaced by $\o{\lambda}$ and $\mu$ replaced by $\o{\mu}$.

\smallskip
\emph{Let $A_k$ be either $P_k(n)$ or $P\!P_k(n)$.}
Let $\tau$ denote the shape of the tableau obtained from $T$ by deleting the
set containing $k$.
If the set containing $k$ is $\{k\}$, then $S$ is obtained from $T$ by deleting
the cell labelled $\{k\}$ and adding an empty cell to the first row.
In this case, $\mu$ is obtained from $\lambda$ by moving a cell to the first row.
If the moved cell came from the first row, then $\o\mu = \o\lambda$,
and otherwise $\o\mu$ is obtained from $\o\lambda$ by deleting a cell.

If the set containing $k$ is not $\{k\}$, then $S$ is the tableau obtained from
$T$ by deleting the cell containing $k$ and inserting a set in the second row
using the RSK insertion procedure. Thus, $\mu$ is obtained from
$\lambda$ by deleting a cell and adding a cell in a row that is not the first
row. If the deleted cell belonged to the first row, then $\o\mu$ is obtained
from $\o\lambda$ by adding a cell. Otherwise, $\o\mu$ is obtained from
$\o\lambda$ by removing a cell and then adding a cell.
This is precisely the branching rule in Table~\ref{tab:branching},
with $\lambda$ replaced by $\o{\lambda}$ and $\mu$ replaced by $\o{\mu}$.
\end{proof}

The map $\phi$ from Proposition~\ref{prop:onestep} allows us to establish
a bijection between standard multiset tableaux and \emph{vacillating tableaux}.
A \defn{vacillating tableau} is a sequence partitions satisfying the condition
$\lambda^{(r)} \vdash n$ and $\lambda^{(r+\frac{1}{2})} \vdash n-1$
with $\lambda^{(r)} \leftarrow \lambda^{(r+\frac{1}{2})}$
and $\lambda^{(r+\frac{1}{2})} \rightarrow \lambda^{(r+1)}$
for $0 \leqslant r < k$ \cite{HalLew, BH}.
A different bijection appears in \cite{BH}. The bijection we provide here is
compatible with the families of tableaux for each of the subalgebras and the
Bratteli diagrams for those subalgebras.

\begin{prop}
\label{prop:bijection-with-vacillating-tableaux}
For each family of subalgebras $A_k$ in Table~\ref{tab:subalgebra}
and for each $\lambda$ a partition of $n\geqslant 2k$,
there is a bijection between $\SMT_{A_k}(\lambda)$ and
the set of vacillating tableaux of the form
\[
    \left((n) = \lambda^{(0)}, \lambda^{(\frac{1}{2})}, \lambda^{(1)}, \lambda^{(1\frac{1}{2})},
    \ldots,\lambda^{(k-\frac{1}{2})}, \lambda^{(k)} = \lambda\right),
\]
where
\[
    \o{\lambda^{(0)}} \Rightarrow \o{\lambda^{(1)}} \Rightarrow \cdots \Rightarrow \o{\lambda^{(k)}}
\]
is a path in the Bratteli diagram for the tower of
algebras $A_0 \subseteq A_1 \subseteq A_2 \subseteq \cdots$.
\end{prop}

\begin{proof}
    Let $T^{(k)}$ be a tableau in $\SMT_{A_k}(\lambda)$.  If $\phi(T^{(k)}) = (S,\tau)$, then set
$T^{(k-1)}=S$, $\lambda^{(k-1)} = \operatorname{shape}(S)$, and $\lambda^{(k-\frac{1}{2})}=\tau$.
Repeat this process on $T^{(k-r)}$ for $k$ steps until $T^{(0)}$ is the
unique empty tableau in $\SMT_{A_0}((n))$.  Record at each step of this process
$\lambda^{(k-1)}$ and $\lambda^{(k-\frac{1}{2})}$.
Now given the sequence of partitions
\[
    \left((n) = \lambda^{(0)}, \lambda^{(\frac{1}{2})}, \lambda^{(1)}, \lambda^{(1\frac{1}{2})},
    \ldots,\lambda^{(k-\frac{1}{2})}, \lambda^{(k)} = \lambda\right)
\]
we can reverse the steps and recover the standard multiset tableau from the sequence.

A consequence of Theorem~\ref{th:edges} is that
the sequence of partitions
$$\o{\lambda^{(0)}} \Rightarrow \o{\lambda^{(1)}} \Rightarrow
\cdots \Rightarrow \o{\lambda^{(k)}}$$
is a path in the Bratteli diagram for the tower of algebras
$A_0 \subseteq A_1 \subseteq A_2 \subseteq \cdots$;
furthermore, for a partition $\lambda \vdash n$ with $n\geqslant 2k$,
the number of these paths that end on the partition $\o\lambda$
is equal to the number of standard multiset tableaux in $\SMT_{A_k}(\lambda)$.
\end{proof}

\begin{example} \label{ex:bijection-with-vacillating-tableaux}
Let $T$ and $S$ be the two tableaux from Example~\ref{ex:brauertab}.
Start with $T^{(9)} = T$. It follows from Example~\ref{ex:brauertab} that $T^{(8)} = S$
and so we record
\begin{center}
    $\lambda^{(9)} = (n-3,2,1)$,
    \ $\lambda^{(8\frac{1}{2})} = (n-4,2,1)$,
    \ and $\lambda^{(8)} = (n-4,2,1,1)$.
\end{center}
The remaining steps of the bijection
are given in Figure~\ref{fig:bijection-with-vacillating-tableaux}.
\begin{figure}[ht!]
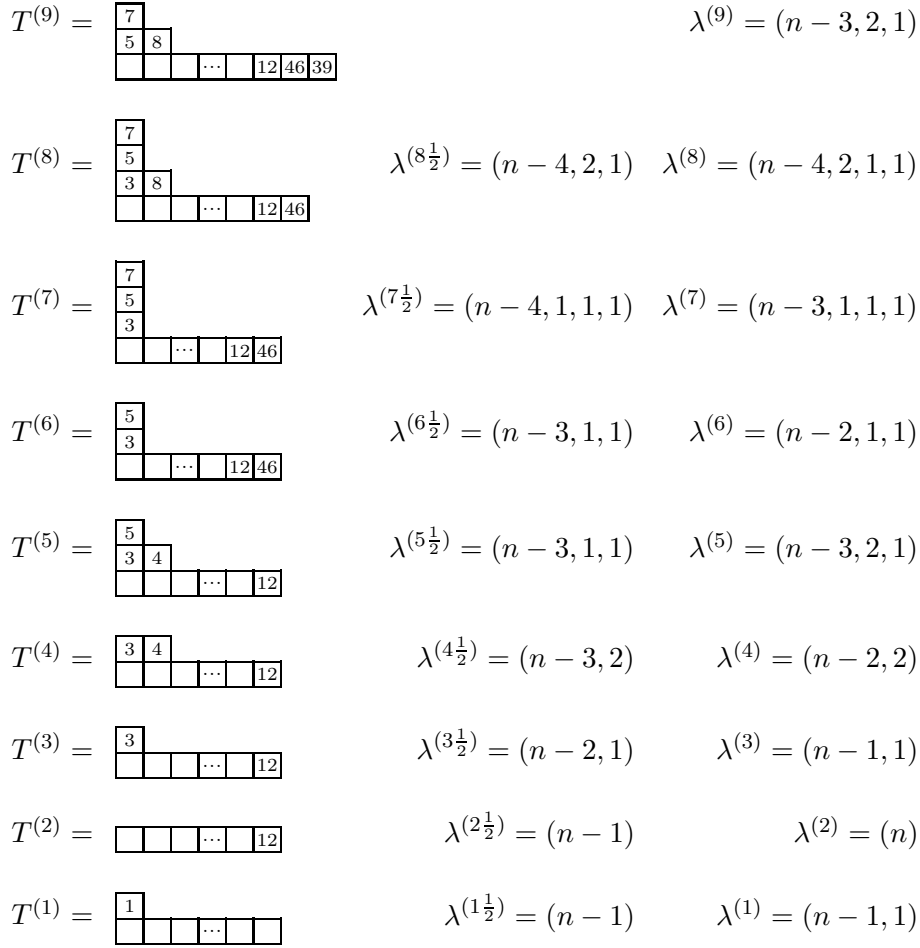

\squaresize=10pt
\begin{tabular}{l r r}
$T^{(9)} =\raisebox{-.25in}{\tiny{ \young{7\cr 5 & 8 \cr & & & \compactcdots& & 12 & 46 & 39 \cr}}}$
&&$\lambda^{(9)} = (n-3,2,1)$\\\\
$T^{(8)} =\raisebox{-.25in}{\tiny{ \young{7\cr 5\cr 3 & 8 \cr & & & \compactcdots& & 12 & 46\cr}}}$
&$\lambda^{(8\frac{1}{2})} = (n-4,2,1)$&$\lambda^{(8)} = (n-4,2,1,1)$\\\\
$T^{(7)} =\raisebox{-.25in}{\tiny{ \young{7\cr 5\cr 3 \cr & & \compactcdots& & 12 & 46\cr}}}$
&$\lambda^{(7\frac{1}{2})} = (n-4,1,1,1)$&$\lambda^{(7)} = (n-3,1,1,1)$\\\\
$T^{(6)} =\raisebox{-.2in}{\tiny{ \young{5\cr 3 \cr & & \compactcdots& & 12 & 46\cr}}}$
&$\lambda^{(6\frac{1}{2})} = (n-3,1,1)$&$\lambda^{(6)} = (n-2,1,1)$\\\\
$T^{(5)} =\raisebox{-.2in}{\tiny{ \young{5\cr 3 & 4 \cr & & & \compactcdots& & 12\cr}}}$
&$\lambda^{(5\frac{1}{2})} = (n-3,1,1)$&$\lambda^{(5)} = (n-3,2,1)$\\\\
$T^{(4)} =\raisebox{-.10in}{\tiny{ \young{3 & 4 \cr & & & \compactcdots& & 12\cr}}}$
&$\lambda^{(4\frac{1}{2})} = (n-3,2)$&$\lambda^{(4)} = (n-2,2)$\\\\
$T^{(3)} =\raisebox{-.10in}{\tiny{ \young{3 \cr & & & \compactcdots& & 12\cr}}}$
&$\lambda^{(3\frac{1}{2})} = (n-2,1)$&$\lambda^{(3)} = (n-1,1)$\\\\
$T^{(2)} =\raisebox{-.05in}{\tiny{ \young{& & & \compactcdots& & 12\cr}}}$
&$\lambda^{(2\frac{1}{2})} = (n-1)$&$\lambda^{(2)} = (n)$\\\\
$T^{(1)} =\raisebox{-.10in}{\tiny{ \young{1\cr& & & \compactcdots& &\cr}}}$
&$\lambda^{(1\frac{1}{2})} = (n-1)$&$\lambda^{(1)} = (n-1,1)$
\end{tabular}
\caption{An example of the bijection from Proposition~\ref{prop:bijection-with-vacillating-tableaux};
see Example~\ref{ex:bijection-with-vacillating-tableaux} for details.}
\label{fig:bijection-with-vacillating-tableaux}
\end{figure}
The corresponding sequence $\{\o{\lambda^{(i)}}\}_{i=1}^{9}$ is the following path in the
Bratteli diagram for the Brauer algebra.
\[
\emptyset \Rightarrow (1) \Rightarrow \emptyset
\Rightarrow (1) \Rightarrow (2) \Rightarrow (2,1)
\Rightarrow (1,1) \Rightarrow (1,1,1)
\Rightarrow (2,1,1) \Rightarrow (2,1).
\]
\end{example}

What we have presented in this section completes the connection
between the results in \cite{HJ} and those in \cite{HalLew}.
The insertion presented in Theorem~\ref{thm:settableaux} is a correspondence
between diagrams and pairs of standard multiset tableaux that motivates the tableaux that arise in
the paper \cite{HJ}.  Theorem~\ref{th:edges} then provides a correspondence
between standard multiset tableaux and paths in the Bratteli diagram.

Since the dimensions of the irreducibles are equal to the number of
paths in the Bratteli diagram, it follows that the number of tableaux
of a given shape is equal to the dimension of the irreducible representation.
This establishes the following result, which can also be proven by enumerating
the tableaux in Lemma~\ref{lemma:tableauxdesc} by a purely combinatorial
argument and verifying that the values agree with Table~\ref{tab:irreducibles}.
\begin{cor}
\label{irred-dim}
Let $n \geqslant 2k$ and $\lambda \vdash n$.
For each of the algebras $A_k$ described in Table~\ref{tab:subalgebra},
let $V^{\o\lambda}_{A_k}$ be the irreducible $A_k$-representation
indexed by $\o\lambda$. Then
\[
\dim\Big(V^{\o\lambda}_{A_k}\Big) = \# \SMT_{A_k}(\lambda).
\]
\end{cor}

\begin{remark}
Benkart and Halverson~\cite{BH} give a bijection between standard multiset tableaux and
vacillating tableaux that is different from the correspondence that we have just
described.  Their bijection does not behave well under restriction to all of the subalgebras
$A_k$ and the corresponding standard multiset tableaux in $\SMT_{A_k}(\lambda)$.
This can be demonstrated via an example.  In $T\!L_2(4)$, there are two diagrams
and two pairs of tableaux (see the third diagram in the second row and the
first diagram in the fourth row of Example~\ref{ex:P2n}), and the
Benkart--Halverson produces the following vacillating tableau
\begin{equation*}
    \squaresize=12pt
    \raisebox{-3pt}{\tiny \young{&&&12\cr}}
    \xrightarrow{\text{~Benkart--Halverson~}}
    \Big(
        \squaresize=10pt
        \raisebox{-2pt}{\young{&&&\cr}\,, \young{&&\cr}\,, \young{&&&\cr}\,, \young{&&\cr}\,, \young{&&&\cr}}\
    \Big),
\end{equation*}
however $(\emptyset, \emptyset, \emptyset)$
is not a path in the Temperley--Lieb Bratteli diagram (see \cite[p. 19]{Jones1}).
On the other hand, our correspondence produces the following vacillating
tableau
\begin{equation*}
    \squaresize=12pt
    \raisebox{-3pt}{\tiny \young{&&&12\cr}}
    \xrightarrow{\text{~Proposition~\ref{prop:bijection-with-vacillating-tableaux}~}}
    \Big(
        \squaresize=10pt
        \raisebox{-2pt}{\young{&&&\cr}\,, \young{&&\cr}\,, \young{\cr&&\cr}\,, \young{&&\cr}\,, \young{&&&\cr}}\
    \Big),
\end{equation*}
and
\squaresize=8pt
$(\emptyset, \young{\cr}\,, \emptyset)$
is a path in the Temperley--Lieb Bratteli diagram.
\end{remark}

\bibliographystyle{amsalpha}
\bibliography{references}

\providecommand{\bysame}{\leavevmode\hbox to3em{\hrulefill}\thinspace}
\providecommand{\MR}{\relax\ifhmode\unskip\space\fi MR }
\providecommand{\MRhref}[2]{%
  \href{http://www.ams.org/mathscinet-getitem?mr=#1}{#2}
}
\providecommand{\href}[2]{#2}
\begin{thebibliography}{CSST10}

\bibitem[BH14]{BH-motzkin}
Georgia Benkart and Tom Halverson, \emph{Motzkin algebras}, European J. Combin.
  \textbf{36} (2014), 473--502. \MR{3131911}

\bibitem[BH19]{BH}
\bysame, \emph{Partition algebras and the invariant theory of the symmetric
  group}, Recent trends in algebraic combinatorics, Assoc. Women Math. Ser.,
  vol.~16, Springer, Cham, 2019, pp.~1--41. \MR{3969570}

\bibitem[BHH17]{BHH}
Georgia Benkart, Tom Halverson, and Nate Harman, \emph{Dimensions of
  irreducible modules for partition algebras and tensor power multiplicities
  for symmetric and alternating groups}, J. Algebraic Combin. \textbf{46}
  (2017), no.~1, 77--108. \MR{3666413}

\bibitem[Bra37]{Brauer}
Richard Brauer, \emph{On algebras which are connected with the semisimple
  continuous groups}, Ann. of Math. (2) \textbf{38} (1937), no.~4, 857--872.
  \MR{1503378}

\bibitem[CLO15]{CLO}
David~A. Cox, John Little, and Donal O'Shea, \emph{Ideals, varieties, and
  algorithms}, fourth ed., Undergraduate Texts in Mathematics, Springer, Cham,
  2015, An introduction to computational algebraic geometry and commutative
  algebra. \MR{3330490}

\bibitem[Com74]{Comtet}
Louis Comtet, \emph{Advanced combinatorics}, enlarged ed., D. Reidel Publishing
  Co., Dordrecht, 1974, The art of finite and infinite expansions. \MR{0460128}

\bibitem[CSST10]{CST}
Tullio Ceccherini-Silberstein, Fabio Scarabotti, and Filippo Tolli,
  \emph{Representation theory of the symmetric groups}, Cambridge Studies in
  Advanced Mathematics, vol. 121, Cambridge University Press, Cambridge, 2010,
  The Okounkov-Vershik approach, character formulas, and partition algebras.
  \MR{2643487}

\bibitem[FHH09]{FHH}
Daniel Flath, Tom Halverson, and Kathryn Herbig, \emph{The planar rook algebra
  and {P}ascal's triangle}, Enseign. Math. (2) \textbf{55} (2009), no.~1-2,
  77--92. \MR{2541502}

\bibitem[Ges95]{Gessel}
Ira~M. Gessel, \emph{Enumerative applications of a decomposition for graphs and
  digraphs}, Discrete Math. \textbf{139} (1995), no.~1-3, 257--271, Formal
  power series and algebraic combinatorics (Montreal, PQ, 1992). \MR{1336842}

\bibitem[GW98]{GoodWall}
Roe Goodman and Nolan~R. Wallach, \emph{Representations and invariants of the
  classical groups}, Encyclopedia of Mathematics and its Applications, vol.~68,
  Cambridge University Press, Cambridge, 1998. \MR{1606831}

\bibitem[Hal01]{Hal}
Tom Halverson, \emph{Characters of the partition algebras}, J. Algebra
  \textbf{238} (2001), no.~2, 502--533. \MR{1823772}

\bibitem[Hal04]{Halverson}
T.~Halverson, \emph{Representations of the q-rook monoid}, J. Algebra
  \textbf{273} (2004), 227--251.

\bibitem[Hal19]{Halverson.2019}
Tom Halverson, \emph{Set-partition tableaux, symmetric group multiplicities,
  and partition algebra modules}, S\'emin. Loth. Combin. \textbf{82B} (2019),
  \#58, 12pp.

\bibitem[Hd14]{HD}
Tom Halverson and Elise delMas, \emph{Representations of the {R}ook-{B}rauer
  algebra}, Comm. Algebra \textbf{42} (2014), no.~1, 423--443. \MR{3169580}

\bibitem[HJ18]{HJ}
Tom {Halverson} and Theodore~N. {Jacobson}, \emph{{Set-partition tableaux and
  representations of diagram algebras}}, arXiv e-prints (2018),
  arXiv:1808.08118.

\bibitem[HL06]{HalLew}
Tom Halverson and Tim Lewandowski, \emph{R{SK} insertion for set partitions and
  diagram algebras}, Electron. J. Combin. \textbf{11} (2004/06), no.~2,
  Research Paper 24, 24. \MR{2195430}

\bibitem[How95]{Howe}
Roger Howe, \emph{Perspectives on invariant theory: {S}chur duality,
  multiplicity-free actions and beyond}, The {S}chur lectures (1992) ({T}el
  {A}viv), Israel Math. Conf. Proc., vol.~8, Bar-Ilan Univ., Ramat Gan, 1995,
  pp.~1--182. \MR{1321638}

\bibitem[HR05]{HalRam}
Tom Halverson and Arun Ram, \emph{Partition algebras}, European J. Combin.
  \textbf{26} (2005), no.~6, 869--921. \MR{2143201}

\bibitem[Inc19]{OEIS}
OEIS~Foundation Inc., \emph{The on-line encyclopedia of integer sequences},
  2019, [Online].

\bibitem[Jon83]{Jones1}
V.~F.~R. Jones, \emph{Index for subfactors}, Invent. Math. \textbf{72} (1983),
  no.~1, 1--25. \MR{696688}

\bibitem[Jon94]{Jones}
\bysame, \emph{The {P}otts model and the symmetric group}, Subfactors
  ({K}yuzeso, 1993), World Sci. Publ., River Edge, NJ, 1994, pp.~259--267.
  \MR{1317365}

\bibitem[Knu70]{Knuth}
Donald~E. Knuth, \emph{Permutations, matrices, and generalized {Y}oung
  tableaux}, Pacific J. Math. \textbf{34} (1970), 709--727. \MR{0272654}

\bibitem[Mac04]{MacMahon}
Percy~A. MacMahon, \emph{Combinatory analysis. {V}ol. {I}, {II} (bound in one
  volume)}, Dover Phoenix Editions, Dover Publications, Inc., Mineola, NY,
  2004, Reprint of {{\i}t An introduction to combinatory analysis} (1920) and
  {{\i}t Combinatory analysis. Vol. I, II} (1915, 1916). \MR{2417935}

\bibitem[Mar90]{Martin1}
Paul~P. Martin, \emph{Representations of graph {T}emperley-{L}ieb algebras},
  Publ. Res. Inst. Math. Sci. \textbf{26} (1990), no.~3, 485--503. \MR{1068862}

\bibitem[MM14]{MM}
Paul Martin and Volodymyr Mazorchuk, \emph{On the representation theory of
  partial {B}rauer algebras}, Q. J. Math. \textbf{65} (2014), no.~1, 225--247.
  \MR{3179659}

\bibitem[MR98]{MarRol}
P.~P. Martin and G.~Rollet, \emph{The {P}otts model representation and a
  {R}obinson-{S}chensted correspondence for the partition algebra}, Compositio
  Math. \textbf{112} (1998), no.~2, 237--254. \MR{1626017}

\bibitem[MS93]{MS}
Paul Martin and Hubert Saleur, \emph{On an algebraic approach to
  higher-dimensional statistical mechanics}, Comm. Math. Phys. \textbf{158}
  (1993), no.~1, 155--190. \MR{1243720}

\bibitem[NPS19]{NPS}
Sridhar {Narayanan}, Digjoy {Paul}, and Shraddha {Srivastava}, \emph{{The
  Multiset Partition Algebra}}, arXiv e-prints (2019), arXiv:1903.10809.

\bibitem[OZ16]{OZ}
Rosa {Orellana} and Mike {Zabrocki}, \emph{{Symmetric group characters as
  symmetric functions}}, arXiv e-prints (2016), arXiv:1605.06672.

\bibitem[OZ19]{OZ2}
\bysame, \emph{{A combinatorial model for the decomposition of multivariate
  polynomials rings as an $S_n$-module}}, arXiv e-prints (2019),
  arXiv:1906.01125.

\bibitem[Rob38]{Robinson}
G.~de~B. Robinson, \emph{On the {R}epresentations of the {S}ymmetric {G}roup},
  Amer. J. Math. \textbf{60} (1938), no.~3, 745--760. \MR{1507943}

\bibitem[Ros00]{Rosas}
Mercedes~Helena Rosas, \emph{A combinatorial overview of the theory of
  {M}ac{M}ahon symmetric functions and a study of the {K}ronecker product of
  {S}chur functions}, ProQuest LLC, Ann Arbor, MI, 2000, Thesis
  (Ph.D.)--Brandeis University. \MR{2700302}

\bibitem[Sch01]{Schur1}
I~Schur, \emph{{U}ber eine {K}lasse von {M}atrizen, die sich einer gegebenen
  {M}atrix zuordnen lassen}, 1901, Dissertation. Berlin. \textbf{76} S JMF
  32.0165.04.

\bibitem[Sch27]{Schur2}
\bysame, \emph{{U}ber die rationalen {D}arstellungen der allgemeinen linearen
  {G}ruppe}, Sitzungsberichte Akad. Berlin (1927), 58--75, JMF 53.0108.05.

\bibitem[Sch61]{Schensted}
C.~Schensted, \emph{Longest increasing and decreasing subsequences}, Canad. J.
  Math. \textbf{13} (1961), 179--191. \MR{0121305}

\bibitem[Sol02]{Solomon}
Louis Solomon, \emph{Representations of the rook monoid}, J. Algebra
  \textbf{256} (2002), no.~2, 309--342. \MR{1939108}

\bibitem[TL71]{TL}
H.~N.~V. Temperley and E.~H. Lieb, \emph{Relations between the ``percolation''
  and ``colouring'' problem and other graph-theoretical problems associated
  with regular planar lattices: some exact results for the ``percolation''
  problem}, Proc. Roy. Soc. London Ser. A \textbf{322} (1971), no.~1549,
  251--280. \MR{0498284}

\bibitem[Wen88]{Wenzl}
Hans Wenzl, \emph{On the structure of {B}rauer's centralizer algebras}, Ann. of
  Math. (2) \textbf{128} (1988), no.~1, 173--193. \MR{951511}

\bibitem[Wes95]{West}
B.~W. Westbury, \emph{The representation theory of the {T}emperley-{L}ieb
  algebras}, Math. Z. \textbf{219} (1995), no.~4, 539--565. \MR{1343661}

\bibitem[Wey97]{Weyl}
Hermann Weyl, \emph{The classical groups}, Princeton Landmarks in Mathematics,
  Princeton University Press, Princeton, NJ, 1997, Their invariants and
  representations, Fifteenth printing, Princeton Paperbacks. \MR{1488158}

\end{thebibliography}

\end{document}